\documentclass[10pt]{amsart}
\usepackage[english]{babel}
\usepackage{amssymb}
\usepackage{amsmath}
\usepackage{lscape}
\usepackage{url}

\newcommand{\ignore}[1]{}

\newtheorem{dummy}{Dummy}

\newtheorem{lemma}[dummy]{Lemma}
\newtheorem{theorem}[dummy]{Theorem}
\newtheorem{proposition}[dummy]{Proposition}
\newtheorem{corollary}[dummy]{Corollary}

\theoremstyle{definition}

\newtheorem{remark}[dummy]{Remark}



\address{Department of Mathematics and Statistics\\
University of Ottawa\\
585 King Edward Avenue
\\ Ottawa, ON K1N 7N5
Canada \\
School of Mathematical Sciences\\
University of Nottingham\\
University Park\\
Nottingham NG7 2RD\\
United Kingdom
}

\author{T. Moran}
\email{tmora083@uottawa.ca}
\author{S. Pumpl\"un}
\email{susanne.pumpluen@nottingham.ac.uk}

\begin{document}

\title[A generalization of the first Tits construction]
{A generalization of the first Tits construction}
\maketitle

\begin{abstract}
Let $F$ be a field of characteristic not 2 or 3. The first Tits construction is a well-known tripling process to construct separable cubic Jordan algebras, especially Albert algebras.
We generalize the first Tits construction by choosing the scalar employed in the tripling process outside of the base field.
This yields a new family of nonassociative unital algebras  which carry a cubic map, and maps that can be viewed as generalized adjoint and
generalized  trace maps. These maps display properties often similar to the ones in the classical setup. In particular, the cubic norm map permits some kind of weak Jordan composition law.
\end{abstract}

\section*{Introduction}

Let $F$ be a field of characteristic not 2 or 3. Separable cubic Jordan algebras over $F$ play an important role in Jordan
theory (where separable means that their trace  defines a nondegenerate bilinear form). It is well known that every separable cubic Jordan algebra  can be obtained by either a first or a second Tits construction \cite[IX, Section 39]{KMRT}. In particular, exceptional simple Jordan algebras, also called Albert algebras, are separable cubic Jordan algebras.
The role of Albert algebras in the structure theory of Jordan algebras is similar to the role of octonion algebras in the structure theory of alternative algebras. Moreover, their automorphism group is an exceptional algebraic group of type $F_4$, their cubic norms have isometry groups of type $E_6$.

In this paper we canonically
generalize the first Tits construction $J(A,\mu)$. The first Tits construction starts with
 a separable associative algebra $A$ of degree three, and uses a scalar $\mu\in F^\times$ in its definition. Our construction also starts with
  $A$, employs the same algebra multiplication as used for the classical first Tits construction, but now allows also $\mu\in A^\times$.

  We obtain a new class of nonassociative unital algebras we again denote by $J(A,\mu)$. They carry a cubic map $N:J(A,\mu)\to A$ that generalizes the classical norm, a map  $T: J(A,\mu) \rightarrow F$ that generalizes the classical trace and a map  $\sharp :J(A,\mu) \rightarrow J(A,\mu)$ that generalizes the classical adjoint of a Jordan algebra.
 Starting with a cubic \'etale algebra $E$,  the algebras obtained this way can be viewed as generalizations of  special nine-dimensiomal Jordan algebras.
 Starting with a
 central simple algebra $A$ of degree three, the algebras obtained this way can be viewed as generalizations of Albert algebras.

Cubic Jordan algebras carry a cubic norm that satisfies some Jordan composition law involving the $U$-operator. Curiously, the cubic map $N:J(A,\mu)\to A$ of our generalized construction still allows some sort of generalized weak Jordan composition law, and some  of the known identities of cubic Jordan algebras involving a generalized trace map and adjoint can be at least partially recovered.

We point out that there already exist canonical nonassociative generalization of cyclic algebras of degree three, involving skew polynomials: the nonassociative cyclic algebras $(K/F,\sigma,\mu)$, where $K/F$ is a cubic separable field extension or a $C_3$-Galois algebra, and $\mu\in K\setminus F$, were first
studied  over finite fields \cite{S}, and then later over arbitrary base fields and rings \cite{St1, St2, BP, Pu0} and applied in space-time block coding \cite{Pu-U, S-Pu-O}. Their ``norm maps'' are isometric to the ``norm maps'' we now obtain of the generalized Tits construction $J(K,\mu)$, and reflect some of the algebra's properties. We show that these algebras are not related, however.

Some obvious questions like when and if the algebras obtained through a generalized first Tits construction are division algebras seem to be very difficult to answer. We will not address these here and only discuss some straightforward implications.

The contents of the paper are as follows:
After introducing the terminology in  Section \ref{sec:preliminaries} and reviewing the classical first Tits construction,  we generalize the classical construction in Section \ref{sec:generalized}  and obtain unital nonassociative algebras with maps that satisfy some of the same identities we know from the classical setup.  We investigate in which special cases several classical identities carry over in Section \ref{sec:identities}.

In Section \ref{sec:example}, we compare the algebras obtained from a generalized first Tits construction starting with a cyclic field extension with the algebras $(K/F,\sigma,\mu)^+$, where $(K/F,\sigma,\mu)$ is a nonassociative cyclic algebra over $F$ of degree three.
If $\mu\in F^\times$ it is well known that these algebras are isomorphic. For $\mu\in K\setminus F$, they are not isomorphic, but their norms are isometric.

This construction was briefly investigated for the first time in Andrew Steele's PhD thesis \cite{St2}. We improved and corrected most of his results, and added many new ones.

%%%%%%%%%%%%%%%%%%%%%%%%%%%%%%%%%%%%%%%%%%%%%%%%%%%%%%%%%%%%%%%%%%%%%%
%
\section{Preliminaries} \label{sec:preliminaries}
%
%%%%%%%%%%%%%%%%%%%%%%%%%%%%%%%%%%%%%%%%%%%%%%%%%%%%%%%%%%%%%%%%%%%%%%%

\subsection{Nonassociative algebras}
Throughout the paper, let $F$ be a field of characteristic not 2 or 3. An $F$-vector space
$A$ is an
\emph{algebra} over $F$ if there exists an $F$-bilinear map $A\times
A\to A$, $(x,y) \mapsto x \cdot y$, denoted simply by juxtaposition
$xy$, the  \emph{multiplication} of $A$. An algebra $A$ is
\emph{unital} if there is an element in $A$, denoted by 1, such that
$1x=x1=x$ for all $x\in A$. We will only consider unital finite-dimensional algebras.

 A nonassociative algebra $A\not=0$ is called a {\it division algebra} if for any $a\in A$, $a\not=0$,
the left multiplication  with $a$, $L_a(x)=ax$,  and the right multiplication with $a$, $R_a(x)=xa$, are bijective.
$A$ is a division algebra if and only if $A$ has no zero divisors. 
The associativity of $A$ is measured by the {\it associator} $[x, y, z] = (xy) z - x (yz)$, and the subalgebras
 ${\rm Nuc}_l(A) = \{ x \in A \, \vert \, [x, A, A]  = 0 \}$,
 ${\rm Nuc}_m(A) = \{ x \in A \, \vert \, [A, x, A]  = 0 \}$, and
 ${\rm Nuc}_r(A) = \{ x \in A \, \vert \, [A,A, x]  = 0 \}$  are called the {\it left nucleus},  {\it middle nucleus} and {\it right nucleus} of $A$.
Their intersection ${\rm Nuc}(A) = \{ x \in A \, \vert \, [x, A, A] = [A, x, A] = [A,A, x] = 0 \}$ is
 the {\it nucleus} of $A$.
The \emph{center} of $A$ is defined as $C(A)=\{ x \in A \, \vert \, xy=yx \text{ for all } y\in A\}\cap {\rm Nuc}(A)$ \cite{Sch}.
All algebras we consider will be unital.

 A nonassociative unital algebra $J$ is called a \emph{cubic Jordan algebra} over $F$ if $J$ is a Jordan algebra, i.e.
$xy = yx$ and $(x^{2}y)x = x^{2}(yx)$ for all $ x,y \in J $, and if its generic minimal polynomial has degree three.
\ignore{%%%%%%%
Every cubic form with adjoint and base point $(N,\sharp,1)$ on a locally free $R$-module $W$ of finite rank
defines a unital Jordan algebra structure $J(N,\sharp,1)$
on $W$ via
$$U_x(y)=T(x,y)x-x^\sharp\times y$$
for all  $x,y\in W$. More precisely, let
$U_{x,y} = U_{x+y} - U_x - U_y$ be the linearization of $U_x$, then
the algebra multiplication  that makes $W$ into a unital algebra is given by $x \circ y=U_{x,y}1$.
}%%%%%%%
An associative algebra $A$ over \(F\), together with the new multiplication
$x \cdot y = \frac{1}{2}(xy+yx)$, is a Jordan algebra over \(F\) denoted by  $ A^{+} $.
A Jordan algebra \(J\) is called \emph{special}, if it is a subalgebra of $ A^{+} $ for some associative algebra \(A\) over \(F\), otherwise \(J\) is \emph{exceptional}. An exceptional Jordan algebra is called an \emph{Albert algebra}.

 The following easy observation  is included for the sake of the reader:

\begin{lemma} \label{A+ division algebra then A}
Let \(A\) be an associative algebra over \(F\) such that $ A^{+} $ is a division algebra. Then \(A\) is a divison algebra.
\end{lemma}

\begin{proof}
Suppose that $ xy = 0 $ for some $ x,y \in A $. Then $ (yx) \cdot (yx) = y(xy)x = 0 $, and since $ A^{+} $ is a division algebra, we obtain $ yx = 0 $. This implies that $x \cdot y = \frac{1}{2}(xy+yx) = 0.$
Using again  that $ A^{+}$ is a division algebra, we deduce that $ x = 0 $ or $ y = 0 $.
\end{proof}

A \emph{nonassociative cyclic algebra} $(K/F,\sigma,c)$ of degree $m$ over $F$ is an $m$-dimensional $K$-vector space
$(K/F,\sigma,c)=K \oplus Kz \oplus K z^2 \oplus\dots \oplus K z^{m-1},$
with multiplication given by the relations
$\label{eq:rule}
z^m=c,~zl=\sigma(l)z$
for all $l\in K$. $(K/F,\sigma,c)$ is a unital algebra with center $F$ and associative if and only if $c\in F^\times$. $(K/F,\sigma,c)$  is a division algebra for all $c\in F^\times$, such that
  $c^s\not\in N_{K/F}(K^\times)$ for all $s$ which are prime divisors of $m$, $1\leq s\leq m-1$.
  If $c\in K\setminus F$ then $(K/F,\sigma,c)$, then $(K/F,\sigma,c)$ is a division algebra for all $c\in K\setminus F$, if $m$ is prime, and for all $c\in K\setminus F$ such that 1, $c,\dots, c^{m-1}$ are linearly independent over $F$  \cite{St1}.

\subsection{Cubic maps}
Let $V$ and $W$ be two finite-dimensional vector spaces over $F$.
A trilinear map $M : V \times  V \times V\rightarrow W$
 is called {\it symmetric} if $M(x,y,z)$ is invariant under all permutations of its variables.
 A map $M:V\rightarrow  W$ is called a {\it cubic map} over $F$, if $M(a x)=a^3 M(x)$ for all $a\in F$, $x\in V$, and if the associated map $M : V \times  V \times V\rightarrow W$ defined by
 $$M(x,y,z)=\frac{1}{6}(M(x+y+z)-M(x+y)-M(x+z)-M(y+z)+ M(x)+M(y)+M(z))$$
  is a (symmetric) $F$-trilinear map. We canonically identify symmetric trilinear maps $ M:V \times  V\times V\to W$ with the
corresponding cubic maps $M:V\rightarrow W$.

A cubic map $M:V\rightarrow F$ is called a {\it cubic form}
and a trilinear map $M: V\times V\times V\rightarrow F$
a {\it trilinear form}  over $F$.
A cubic map  is called {\it nondegenerate} if $v = 0$ is the only vector such
that $M (v, v_{2}, , v_3) = 0$ for all $v_i \in V$.  A cubic map $M$ is called {\it anisotropic} if
$M(x)=0$ implies that $x=0$, else {\it isotropic}.
For a nonassociative algebra $A$ over $F$ together with a nondegenerate cubic form $M:A\rightarrow F$,  $M$ is called
\emph{multiplicative}, if $M(vw)=M(v)M(w)$ for all $v, w\in A$.

\subsection{Algebras of degree three} \label{Degree 3 algebras}(cf. for instance \cite{KMRT}, \cite[Chapter C.4]{McC})

Let \(A\) be a unital separable associative algebra over \(F\) with  norm  $ N_{A} : A \rightarrow F $. 
Let $ x,y \in A $ and let \(Z\) be an indeterminate. The linearisation
$
N_{A}(x+Zy) = N_{A}(x) + ZN_{A}(x;y) + Z^{2}N_{A}(y;x) + Z^{3}N_{A}(y)
$
of $ N_{A} $, i.e. the coefficient of \(Z\) in the above expansion,
 is quadratic in \(x\) and linear in \(y\), and is denoted by $ N_{A}(x;y) $.
 Indeed, we have
\begin{equation*}
N_{A}(x+Zx) = N_{A}((1+Z)x) = (1+Z)^{3}N_{A}(x) = (1+3Z+3Z^{2}+Z^{3})N_{A}(x),
\end{equation*}
so $ N_{A}(1;1) = 3N_{A}(1) = 3 $.
Linearize $ N_{A}(x;y) $  to obtain a symmetric trilinear map $N_A:A\times A\times A \rightarrow F$,
$N_{A}(x,y,z) = N_{A}(x+z;y) - N_{A}(x;y) - N_{A}(z;y).$
We define
\begin{align*}
T_{A}(x) &= N_{A}(1;x), \\
T_{A}(x,y) &= T_{A}(x)T_{A}(y) - N_{A}(1,x,y), \\
S_{A}(x) &= N_{A}(x;1), \\
x^{\sharp } &= x^{2} - T_{A}(x)x+S_{A}(x)1,
\end{align*}
for all $ x,y \in A $. We call $ x^{\sharp } $ the \emph{adjoint} of \(x\), and define the \emph{sharp map} $ \sharp  :A\times A \rightarrow A$,
 $x\sharp y = (x+y)^{\sharp }-x^{\sharp }-y^{\sharp }$ as the linearisation of the adjoint.
We observe that $T_{A}(1) = S_{A}(1) = 3$. Since the trilinear map $ N_{A}(x,y,z) $ is symmetric,
\begin{equation} \label{Tr = Tr}
T_{A}(x,y) = T_{A}(y,x)
\end{equation}
for  all $ x,y \in A $.

 The associative algebra \(A\) is called of \emph{degree three} or a \emph{cubic algebra}, if the following three identities are satisfied  for all $ x,y \in A $:
\begin{align}
& x^{3}-T_{A}(x)x^{2}+S_{A}(x)x-N_{A}(x)1 = 0 \mbox{ (\emph{degree} } 3 \mbox{ \emph{identity})}, \label{degree three  identity} \\
&T_{A}(x^{\sharp },y) = N_{A}(x;y) \mbox{ \emph{(trace-sharp formula})}, \label{Trace-sharp formula} \\
& T_{A}(x,y) = T_{A}(xy) \mbox{ \emph{(trace-product formula).}} \label{Trace-product formula}
\end{align}

 For the rest of Section \ref{Degree 3 algebras}, we assume that \(A\) is a separable algebra of degree three over \(F\) with  cubic norm  $ N_{A}: A \rightarrow F $. Note that  (\ref{degree three  identity}) is equivalent to the condition that
\begin{equation} \label{degree three  identity equivalent}
xx^{\sharp } = x^{\sharp }x = N_{A}(x)1,
\end{equation}
and combining (\ref{Tr = Tr}) with (\ref{Trace-product formula}) gives
\begin{equation} \label{Tr(xy) = Tr(yx)}
T_{A}(xy) = T_{A}(yx).
\end{equation}
An element $ x \in A $ is invertible if and only if $ N_{A}(x) \neq 0 $.
The inverse of $ x \in A $ is $ x^{-1} = N_{A}(x)^{-1}x^{\sharp }$.
It can be shown that
\begin{equation} \label{xy hash = y hash x hash}
(xy)^{\sharp } = y^{\sharp }x^{\sharp }
\end{equation}
for all $ x,y \in A $. Notice that
\begin{equation} \label{use it twice}
T_{A}(x^{\sharp }) = T_{A}(x^{\sharp },1) = N_{A}(x;1) = S_{A}(x),
\end{equation}
 using (\ref{Trace-product formula}) and  (\ref{Trace-sharp formula}). We also have $ S_{A}(x) = T_{A}(x^{\sharp }) = T_{A}(x^{2})-T_{A}(x)^{2}+3S_{A}(x) $, so
\begin{equation} \label{S(x) = Tr - Tr}
2S_{A}(x) = T_{A}(x)^{2}-T_{A}(x^{2}).
\end{equation}
A straightforward calculation shows that
\begin{equation*} \label{sharp alternative}
x\sharp  y = 2(x \cdot y) -T_{A}(x)y -T_{A}(y)x+(T_{A}(x)T_{A}(y)-T_{A}(x\cdot y))1.
\end{equation*}
for all $ x,y \in A $.
In particular,
\[x \cdot y = \frac{1}{2}(xy+yx)= \frac{1}{2}(x \sharp y + T_A(x)y + T_A(y)x - (T_A(x)T_A(y)- T_A(x,y))1)\]
for all $x,y \in A$
and by employing (\ref{degree three  identity equivalent}) and the adjoint identity in \(A\) we see that the norm $N_A$ satisfies the relation
\begin{equation}\label{normsharp} 
N_A(x^\sharp) = N_A(x)^2.
\end{equation}

 \(A^+\) satisfies the \emph{adjoint identity}
$(x^{\sharp })^{\sharp } = N_{A}(x)x$ \label{adjoint identity lemma}
for all $ x \in A $.
By (\ref{normsharp}) we have
$N_{A}(x^{\sharp })1 = x^{\sharp }(x^{\sharp })^{\sharp } = x^{\sharp }N_{A}(x)x = N_{A}(x)^{2}1.$
For  $ x,y \in A $, we define the operators $ U_{x}: A \rightarrow A $,
$U_{x}(y) = T_{A}(x,y)x-x^{\sharp }\sharp  y$
and  $ U_{x,y}: A \rightarrow A $,
$U_{x,y}(z) = U_{x+y}(z)-U_{x}(z)-U_{y}(z).$ Then  we have $ x \cdot y = \frac{1}{2}U_{x,y}(1) $ for all $ x,y \in A $
and
\ignore{%%%%%%
\begin{proposition}
For each $ x,y \in A $, we have $ x \cdot y = \frac{1}{2}U_{x,y}(1) $.
\end{proposition}

\begin{proof}
By Lemma \ref{sharp alternative}, $ x^{\sharp }\sharp  1 = T_{A}(x^{\sharp })1-x^{\sharp } $. This gives that
\begin{equation*}
U_{x}(1) = T_{A}(x,1)x-x^{\sharp }\sharp  1 = T_{A}(x)x-T_{A}(x^{\sharp })1+x^{\sharp },
\end{equation*}
where in the second equality we have used the trace-product formula (see (\ref{Trace-product formula})). So
\begin{align}
U_{x,y}(1) &= U_{x+y}(1)-U_{x}(1)-U_{y}(1) \nonumber \\ &= T_{A}(x+y)(x+y)-T_{A}((x+y)^{\sharp })1+(x+y)^{\sharp }-T_{A}(x)x+T_{A}(x^{\sharp })1 -x^{\sharp } \nonumber \\ &\hspace{0.4cm}-T_{A}(y)y+T_{A}(y^{\sharp })1-y^{\sharp } \nonumber \\ &= T_{A}(x)y+T_{A}(y)x+x\sharp  y - T_{A}(x\sharp  y)1. \label{equation in this prop}
\end{align}
Using Lemma \ref{sharp alternative}, the linearity of $ T_{A} $ and (\ref{Tr(xy) = Tr(yx)}), we obtain after some simplification that (\ref{equation in this prop}) is equal to $ 2(x\cdot y) $. Hence $ x \cdot y = \frac{1}{2}U_{x,y}(1) $.
\end{proof}
}%%%%%
\begin{equation} \label{McCrimmon xyx equation}
xyx = T_{A}(x,y)x-x^{\sharp }\sharp  y,
\end{equation}
  hence $U_{x}(y) = xyx$
for all $ x,y \in A^\times $.

Define
$x \times y = \frac{1}{2}(x \sharp  y),$ and $ \bar{x} = \frac{1}{2}(T_{A}(x)1-x)$
for $x,y\in A$.
(Note that some literature does not include the factor  $\frac{1}{2} $ in the definition of $\times$, e.g. \cite{Thakur}.)
By Lemma \ref{sharp alternative},  we then have
$$x \times y = x\cdot y -\frac{1}{2}T_{A}(x)y - \frac{1}{2}T_{A}(y)x+ \frac{1}{2}(T_{A}(x)T_{A}(y)-T_{A}(x\cdot{y}))1$$
for all $x,y \in A $, and hence
\begin{equation} \label{axa = a adj}
x \times x = x^{2}-T_{A}(x)x+\frac{1}{2}(T_{A}(x)^{2}-T_{A}(x^{2})) = x^{\sharp },
\end{equation}
 using (\ref{S(x) = Tr - Tr}).

\subsection{The first Tits construction} \label{The first Tits construction}

Let $A$ be an separable
associative algebra of degree three  over $F$ with norm $N_A$, trace $T_A$ and adjoint map $\sharp$. Let $\mu \in F^{\times}$ and
define the $F$-vector space $J= J(A, \mu) = A_0 \oplus A_1 \oplus A_2$, where  $A_i=A$ for $i=0,1,2$.
Then $J(A,\mu)$ together with the multiplication
$$
(x_{0},x_{1},x_{2})(y_{0},y_{1},y_{2}) = (x_{0}\cdot y_{0}+\overline{x_{1}y_{2}}+\overline{y_{1}x_{2}}, \overline{x_{0}}y_{1}+\overline{y_{0}}x_{1}+\mu^{-1}(x_{2}\times y_{2}), x_{2}\overline{y_{0}}+y_{2}\overline{x_{0}}+\mu(x_{1}\times y_{1}))
$$
becomes a separable cubic Jordan algebra over $F$.  $J(A,\mu)$ is called a \emph{first Tits construction}.
 $A^+$  is a subalgebra of $J(A,\mu)$ by canonically identifying it with $A_0$. If $A$ is a cubic etale algebra, then $J(A,\mu)\cong D^+$ for
 with $D$ an associative cyclic algebra $D$ of degree three.
 If $A$ is a central simple algebra of degree three then $J(A,\mu)$ is an Albert algebra.

We define the \emph{cubic norm form} $ N: J(A,\mu) \rightarrow F $, the \emph{trace} $ T: J(A,\mu) \rightarrow F $, and the quadratic map $\sharp:J(A, \mu )\rightarrow  J(A, \mu )$  (the \emph{adjoint}) by
\begin{align*}
N((x_0, x_1, x_2))& = N_A(x_0) +  \mu  N_A(x_1) +  \mu^{-1} N_A(x_2) - T_{A}(x_0x_1x_2) \label{Cubic norm form N}\\
T((x_0, x_1, x_2)) &= T_{A}(x_{0}), \\
(x_0, x_1, x_2)^{\sharp } &= (x_{0}^{\sharp }-x_{1}x_{2},\mu^{-1}x_{2}^{\sharp }-x_{0}x_{1},\mu x_{1}^{\sharp }-x_{2}x_{0}).
\end{align*}

The \emph{intermediate quadratic form} $S:J(A,\mu) \rightarrow F$, $S(x_0) = N(x;1)$, linearizes to a
 map $S: J(A,\mu) \times J(A,\mu)\rightarrow F$.
The \emph{sharp map} $ \sharp  :J(A,\mu)\times J(A,\mu)\rightarrow J(A,\mu)$ is the linearisation
$x\sharp  y = (x+y)^{\sharp }-x^{\sharp }-y^{\sharp }$
of the adjoint. For every $ x = (x_{0},x_{1},x_{2}) \in J(A,\mu) $, we have
$ x\sharp  1 = T(x)1-x  \label{x sharp 1 in F}$
and
\begin{equation*}
x\sharp y = (x_{0}\sharp y_{0}-x_{1}y_{2}-y_{1}x_{2}, \mu^{-1}(x_{2}\sharp y_{2})-x_{0}y_{1}-y_{0}x_{1},\mu(x_{1}\sharp y_{1})-x_{2}y_{0}-y_{2}x_{0})
\end{equation*}
for all $ x = (x_{0},x_{1},x_{2}), y = (y_{0},y_{1},y_{2}) \in J(A,\mu) $.
We  define the \emph{trace symmetric bilinear form} $T: J(A,\mu)\times J(A,\mu)\rightarrow F $,
$T(x,y) = T_{A}(x_{0}y_{0}) + T_{A}(x_{1}y_{2}) + T_{A}(x_{2}y_{1}).$
Then for all $ x,y \in J(A,\mu) $, we have
\begin{equation}
T(x,y) = T(xy) . \label{bilinear trace equals}
\end{equation}

\begin{remark} $(N,\sharp,1)$ is a cubic form with adjoint and base point $(1,0,0)$ on $J(A, \mu )$ which makes $J(A, \mu )$ into a cubic Jordan algebra $J(N,\sharp,1)$.
\end{remark}

%%%%%%%%%%%%%%%%%%%%%%%%%%%%%%%%%%%%%%%%%%%%%%%%%%%%%%%%%%%%%%%%%%%%%%
%
\section{The generalised first Tits construction}\label{sec:generalized}
%
%%%%%%%%%%%%%%%%%%%%%%%%%%%%%%%%%%%%%%%%%%%%%%%%%%%%%%%%%%%%%%%%%%%%%%

Let $A$ be a separable
associative algebra of degree three over $F$ with norm $N_A$, trace $T_A$ and adjoint map $\sharp $.

\subsection{The algebra $J(A, \mu)$}
We now  generalise the first Tits construction by choosing the scalar $\mu\in A^\times$.
Then the $F$-vector space $J(A, \mu) = A_0 \oplus A_1 \oplus A_2$, where again $A_i=A$ for $i=0,1,2$,
 becomes a unital nonassociative algebra over $F$ together with the multiplication given by
$$
(x_{0},x_{1},x_{2})(y_{0},y_{1},y_{2}) = (x_{0}\cdot y_{0}+\overline{x_{1}y_{2}}+\overline{y_{1}x_{2}}, \overline{x_{0}}y_{1}+\overline{y_{0}}x_{1}+\mu^{-1}(x_{2}\times y_{2}), x_{2}\overline{y_{0}}+y_{2}\overline{x_{0}}+\mu(x_{1}\times y_{1})).
$$
The algebra $J(A,\mu)$ is called a \emph{generalized first Tits construction}. The special Jordan algebra $A^+$  is a subalgebra of $J(A,\mu)$ by canonically identifying it with $A_0$. If $\mu  \in F^{\times} $, then $ J(A,\mu) $ is the first Tits construction from Section \ref{The first Tits construction}.

 We define a \emph{(generalized) cubic norm map} $ N: J(A,\mu) \rightarrow A $, a \emph{(generalized) trace} $ T: J(A,\mu) \rightarrow F $, and a quadratic map  $\sharp:J(A, \mu )\rightarrow  J(A, \mu )$  via
\begin{align}
N((x_0, x_1, x_2))& = N_A(x_0) +  \mu  N_A(x_1) +  \mu^{-1} N_A(x_2) - T_{A}(x_0x_1x_2) \label{Cubic norm form N}\\
T((x_0, x_1, x_2)) &= T_{A}(x_{0}), \\
(x_0, x_1, x_2)^{\sharp } &= (x_{0}^{\sharp }-x_{1}x_{2},\mu^{-1}x_{2}^{\sharp }-x_{0}x_{1},\mu x_{1}^{\sharp }-x_{2}x_{0}).
\end{align}
Put $ \sharp  : J(A,\mu)\times J(A,\mu) \rightarrow J(A,\mu)$,
$x\sharp  y = (x+y)^{\sharp }-x^{\sharp }-y^{\sharp }$, then it can be verified by a direct computation that
\begin{equation*}
x\sharp y = (x_{0}\sharp y_{0}-x_{1}y_{2}-y_{1}x_{2}, \mu^{-1}(x_{2}\sharp y_{2})-x_{0}y_{1}-y_{0}x_{1},\mu(x_{1}\sharp y_{1})-x_{2}y_{0}-y_{2}x_{0})
\end{equation*}
for all $ x = (x_{0},x_{1},x_{2}), y = (y_{0},y_{1},y_{2}) \in J(A,\mu) $.
We also define a  symmetric $F$-bilinear form  $ T: J(A,\mu)\times J(A,\mu)\rightarrow  F$ via
$
T(x,y) = T_{A}(x_{0}y_{0}) + T_{A}(x_{1}y_{2}) + T_{A}(x_{2}y_{1}).
$

The  quadratic form $S_A:A \rightarrow F$, $S_A(x_0) = N_A(x;1)$, linearizes to $S_A: A \times A \rightarrow F$, and we have $S_A(x_0) = T_A(x_0^\sharp)$ for all $x_0 \in A$.
 We extend $S_A$ to $J(A,\mu)$ by defining the quadratic map $S:J(A,a)\rightarrow A$, $S(x) = N(x;1)$.
As in the classical case we obtain:

\begin{theorem} \label{id1}
(i) \cite[Proposition 5.2.2]{St2} For all $ x \in J(A,\mu) $, we have $S(x)= T(x^\sharp)$ and the linearization $S: J(A,\mu) \times J(A,\mu) \rightarrow A$ satisfies
\[S(x,y) = T(x)T(y) - T(x,y)\]
for all $y \in J(A,\mu)$.
\\
(ii) \cite[Lemma 5.2.3]{St2} For all $ x,y \in J(A,\mu) $, we have $ T(x,y) = T(xy) $.
\\ (iii) \cite[Lemma 5.2.3]{St2} For all $ x \in J(A,\mu) $, we have $ x\sharp  1 = T(x)1-x $.  
\end{theorem}

\begin{proof}
(i) 
Let $x = (x_0, x_1, x_2), y = (y_0, y_1, y_2) \in J(A,a)$, then
$$
N(x; y) = N_A(x_0; y_0) +  \mu  N_A(x_1; y_1) +  \mu^{-1}N_A(x_2; y_2)
 - T_A(x_0x_1y_2) - T_A(x_0y_1x_2) - T_A(y_0x_1x_2),
$$
and since $S(x) = N(x; 1)$ we get
$S(x) =  N_A(x_0; 1) - T_A(x_1x_2) = S_A(x_0) - T_A(x_1x_2).$
On the other hand,
$$
T(x^\sharp)  = T_A(x_0^\sharp - x_1x_2) = T_A(x_0^\sharp) - T_A(x_1x_2)
= S_A(x_0) - T_A(x_1x_2) = S(x).
$$
We have  $S_A(x_0,y_0) = T_A(x_0)T_A(y_0) - T_A(x_0,y_0)$ for all $x_0, y_0 \in A$. Linearizing $S$ gives
$
S(x,y) = S_A(x_0, y_0) - T_A(x_1y_2) - T_A(y_1x_2)= T_A(x_0)T_A(y_0) - T_A(x_0, y_0) - T_A(x_1y_2) - T_A(y_1x_2)= T(x)T(y) - T(x,y)$
using the definitions of $T_A(x_i)$ and $T_A(x_i, y_i)$ and the fact that $T_A(x_0, y_0) = T_A(x_0y_0)$.
\\
(ii) Let $ x = (x_{0},x_{1},x_{2}), y = (y_{0},y_{1},y_{2})\in J(A,\mu)$. Since $ T_{A} $ is linear, we get
\begin{align*}
T(xy) &= T_{A}(x_{0}\cdot{y_{0}}) + T_{A}(\overline{x_{1}y_{2}})+T_{A}(\overline{y_{1}x_{2}}) \\ &= \frac{1}{2}(T_{A}(x_{0}y_{0})+T_{A}(y_{0}x_{0})) + \frac{1}{2}(T_{A}(x_{1}y_{2})T_{A}(1) - T_{A}(x_{1}y_{2})) \\ &\hspace{0.4cm}+\frac{1}{2}(T_{A}(y_{1}x_{2})T_{A}(1)-T_{A}(y_{1}x_{2})).
\end{align*}
By (\ref{Tr(xy) = Tr(yx)}) we get $ T_{A}(x_{0}y_{0}) = T_{A}(y_{0}x_{0}) $ and $ T_{A}(y_{1}x_{2}) = T_{A}(x_{2}y_{1}) $. Since we have $ T_{A}(1) = 3 $ we get
$T(xy) = T_{A}(x_{0}y_{0}) + T_{A}(x_{1}y_{2}) + T_{A}(x_{2}y_{1}) = T(x,y). $
\\ (iii) Let $ x = (x_{0},x_{1},x_{2}) \in J(A,\mu)$. By Lemma \ref{sharp alternative}, we have $ x_{0}\sharp  1 = T_{A}(x_{0})1-x_{0} $. Thus
$x\sharp  1 = (x_{0}\sharp  1,-x_{1},-x_{2}) = T(x)1 - x. $
\end{proof}

\begin{theorem} \label{some relations}
Let $ \mu \in A^{\times} $, and let $ x,y \in J(A,\mu) $. Then
\\ (i) $x^\sharp=x^2-T(x)x+S(x)1$,
\\ (ii) $S(x)=T(x^\sharp)$,
\\ (iii) $T(x\times y)=\frac{1}{2}(T(x)T(y)-T(xy))$.
\end{theorem}

Note that these are relations that also hold for a  cubic form with adjoint and base point
 $(N,\sharp,1)$ \cite{McC, PR}.

\begin{proof}
Let $ x = (x_{0},x_{1},x_{2}), y = (y_{0},y_{1},y_{2})\in J(A,\mu) $.
\\ (i) We have that
$x^{2} - T(x)x+S(x)1 = (x_{0},x_{1},x_{2})^{2}-T_{A}(x_{0})(x_{0},x_{1},x_{2})+(S_{A}(x_{0})1-T_{A}(x_{1}x_{2})1,0,0) = (x_{0}^{2}-T_{A}(x_{0})x_{0}+S_{A}(x_{0})1 + 2\overline{x_{1}x_{2}}-T_{A}(x_{1}x_{2})1, \mu^{-1}x_{2}^{\sharp}+2\overline{x_{0}}x_{1}-T_{A}(x_{0})x_{1},\mu x_{1}^{\sharp}+2x_{2}\overline{x_{0}}-T_{A}(x_{0})x_{2}) = (x_{0}^{\sharp}-x_{1}x_{2},\mu^{-1}x_{2}^{\sharp}-x_{0}x_{1}, \mu x_{1}^{\sharp}-x_{2}\overline{x_{0}}) = x^{\sharp}.
$
\\ (ii) As for the classical construction,
\begin{equation*}
T(x^{\sharp}) = T_{A}(x_{0}^{\sharp}-x_{1}x_{2}) = T_{A}(x_{0}^{\sharp}) - T_{A}(x_{1}x_{2}) = S_{A}(x_{0})-T_{A}(x_{1}x_{2}) = S(x).
\end{equation*}
(iii) Since $ x\times y = \frac{1}{2}(x \sharp y) = \frac{1}{2}(x_{0}\#y_{0}-x_{1}y_{2}-y_{1}x_{2}, \mu^{-1}(x_{2}\#y_{2})-x_{0}y_{1}-y_{0}x_{1},\mu(x_{1}\#y_{1})-x_{2}y_{0}-y_{2}x_{0}) $, we obtain
$
T(x \times y) = T_{A}(x_{0} \times y_{0})-\frac{1}{2}T_{A}(x_{1}y_{2})-\frac{1}{2}T_{A}(y_{1}x_{2}) = \frac{1}{2}(T_{A}(x_{0})T_{A}(y_{0})-T_{A}(x_{0}y_{0})-T_{A}(x_{1}y_{2})-T_{A}(y_{1}x_{2})) = \frac{1}{2}(T(x)T(y)-T(xy)).$
\end{proof}

Define  operators $ U_{x}, U_{x,y}: J(A,\mu) \rightarrow J(A,\mu) $ via
\begin{equation*}
U_{x}(y) = T(x,y)x-x^{\sharp }\sharp  y, \quad U_{x,y}(z) = U_{x+y}(z)-U_{x}(z)-U_{y}(z)
\end{equation*}
for all $ z \in J(A,\mu) $.

\begin{proposition} (cf. \cite[Proposition 5.2.4]{St2} without factor $\frac{1}{2}$ because of slightly different terminology)
For all $ x,y \in J(A,\mu) $, we have $ xy = \frac{1}{2}U_{x,y}(1) $.
\end{proposition}

This generalizes the classical setup. Our proof is different to the one of \cite[Proposition 5.2.4]{St2}, which also proves this result without the factor $\frac{1}{2}$ because of the slightly different definition of the multiplication.

\begin{proof}
We  find that $ U_{x}(1) = T(x,1)x-x^{\sharp }\sharp  1 = T(x)x-T(x^{\sharp })1+x^{\sharp } $, where in the second equality we have used Theorem \ref{id1} and the fact that $ T(x,1) = T(x) $ by Theorem \ref{id1}. So
\begin{align*}
U_{x,y}(1) &= U_{x+y}(1)-U_{x}(1)-U_{y}(1) \\ &= T(x+y)(x+y)-T((x+y)^{\sharp })1+(x+y)^{\sharp }-T(x)x+T(x^{\sharp })1 -x^{\sharp } \\ &\hspace{0.4cm}-T(y)y+T(y^{\sharp })1-y^{\sharp } \\ &= T(x)y+T(y)x+x\sharp  y - T(x\sharp  y)1.
\end{align*}
We look at the first component of $ xy $ and $ U_{x,y}(1) $: let $ x = (x_{0},x_{1},x_{2}) $ and $ y = (y_{0},y_{1},y_{2}) $. Then the first component of $ U_{x,y}(1) = T(x)y+T(y)x+x\sharp  y - T(x\sharp  y)1 $ is
\begin{equation} \label{equation in middle of proof for this one}
T_{A}(x_{0})y_{0}+T_{A}(y_{0})x_{0}+x_{0}\sharp  y_{0}-x_{1}y_{2}-y_{1}x_{2} - T_{A}(x_{0}\sharp  y_{0}-x_{1}y_{2}-y_{1}x_{2})1.
\end{equation}
Using Lemma \ref{sharp alternative}, the linearity of $ T_{A} $ and (\ref{Tr(xy) = Tr(yx)}), we obtain after some simplification that (\ref{equation in middle of proof for this one}) is equal to
\begin{equation*}
2(x_{0} \cdot y_{0})+T_{A}(x_{1}y_{2})-x_{1}y_{2}+T_{A}(y_{1}x_{2})-y_{1}x_{2} = 2(x_{0}\cdot y_{0}) + 2 \hspace{0.05cm}\overline{x_{1}y_{2}}+2\hspace{0.05cm}\overline{y_{1}x_{2}}.
\end{equation*}
This is  equal to \(2\) times the first component of $ xy $. Now we look at the second component of $ xy $ and $ U_{x,y}(1) $: the second component of $ U_{x,y}(1) = T(x)y+T(y)x+x\sharp  y - T(x\sharp  y)1 $ is
\begin{equation*}
T_{A}(x_{0})y_{1}+T_{A}(y_{0})x_{1}+\mu^{-1}(x_{2}\sharp  y_{2})-x_{0}y_{1}-y_{0}x_{1} = 2\overline{x_{0}}y_{1}+2\overline{y_{0}}x_{1}+2\mu^{-1}(x_{2} \times y_{2}).
\end{equation*}
This is precisely equal to \(2\) times the second component of $ xy $. Finally, the third component of $ 2xy $ and $ U_{x,y}(1) $ are equal, too: The third component of $ U_{x,y}(1) = T(x)y+T(y)x+x\sharp  y - T(x\sharp  y)1 $ is
\begin{equation*}
T_{A}(x_{0})y_{2}+T_{A}(y_{0})x_{2}+\mu(x_{1}\sharp  y_{1})-x_{2}y_{0}-y_{2}x_{0} = 2x_{2}\overline{y_{0}}+2y_{2}\overline{x_{0}}+2\mu (x_{1} \times y_{1}).
\end{equation*}
This is precisely equal to \(2\) times the third component of $ xy $.
\end{proof}

\begin{theorem}
If $ \mu \in A^{\times} $ and $ A \neq F $, then $ \operatorname{Nuc}_{l}(J(A,\mu)) = \operatorname{Nuc}_{r}(J(A,\mu)) = F $.
\end{theorem}

\begin{proof}
 Let $ (x_{0},x_{1},x_{2}) \in \operatorname{Nuc}_{l}(J(A,\mu)) $, then
 $$ (x_{0},x_{1},x_{2})[(0,1,0)(0,0,1)] = [(x_{0},x_{1},x_{2})(0,1,0)](0,0,1) $$  implies that
$
(x_{0},x_{1},x_{2}) = (\overline{\overline{x_{0}}},\mu^{-1}(\overline{\mu \overline{x_{1}}}), \overline{\overline{x_{2}}}).
$, that means $ x_{0} = \overline{\overline{x_{0}}} $ and $ x_{2} = \overline{\overline{x_{2}}} $. Using the definition of $ \overline{\overline{x_{0}}} $, we obtain $ \overline{\overline{x_{0}}} = \frac{1}{4}(\operatorname{T}_{A}(x_{0})+x_{0}) $, so $ x_{0} = \frac{1}{4}(\operatorname{T}_{A}(x_{0})+x_{0}) $. Thus $ x_{0} = \frac{1}{3}\operatorname{T}_{A}(x_{0}) \in F $. Also, since $ x_{2} = \overline{\overline{x_{2}}} $, we find in a similar way  that $ x_{2} = \frac{1}{3}\operatorname{T}_{A}(x_{2}) \in F $. Next, since $ x = (x_{0},x_{1},x_{2}) \in \operatorname{Nuc}_{l}(J(A,\mu)) $, we have that $ (x_{0},x_{1},x_{2})[(0,0,1)(0,1,0)] = [(x_{0},x_{1},x_{2})(0,0,1)](0,1,0) $. This implies that
$
(x_{0},x_{1},x_{2}) = (\overline{\overline{x_{0}}}, \overline{\overline{x_{1}}}, \mu (\overline{\mu^{-1}\overline{x_{2}}})),
$
and so $ x_{1} = \overline{\overline{x_{1}}} $. We now find in a similar way  that $ x_{1} = \frac{1}{3}\operatorname{T}_{A}(x_{1}) \in F $, thus
$ \operatorname{Nuc}_{l}(J(A,\mu)) \subseteq \{(x_{0},x_{1},x_{2}) \in J \hspace{0.1cm} | \hspace{0.1cm} x_{0},x_{1},x_{2} \in F \} $.
 Let $ x = (x_{0},x_{1},x_{2}) \in \operatorname{Nuc}_{l}(J(A,\mu)) $, and let $ a \in A \setminus F $. Then $ (x_{0},x_{1},x_{2})[(0,0,1)(0,a,0)] = [(x_{0},x_{1},x_{2})(0,0,1)](0,a,0) $ which implies that
$
(x_{0} \cdot \overline{a}, \overline{\overline{a}}x_{1}, x_{2}\overline{\overline{a}}) = (\overline{a\overline{x_{0}}}, \overline{\overline{x_{1}}}a, \mu(\mu^{-1}\overline{x_{2}} \times a)),$
and so $ \overline{\overline{a}}x_{1} = \overline{\overline{x_{1}}}a $. Assume towards a contradiction that $ x_{1} \neq 0 $. Since $ x_{1} \in F $, $ x_{1} $ is invertible and $ \overline{\overline{x_{1}}} = x_{1} $. Thus the condition $ \overline{\overline{a}}x_{1} = \overline{\overline{x_{1}}}a $ yields $ a = \overline{\overline{a}} $, and so $ a = \frac{1}{3}\operatorname{T}_{A}(a) \in F $ which is a contradiction. Next, since $ (x_{0},x_{1},x_{2}) \in \operatorname{Nuc}_{l}(J(A,\mu)) $, we know that $ (x_{0},x_{1},x_{2})[(0,1,0)(0,0,a)] = [(x_{0},x_{1},x_{2})(0,1,0)](0,0,a) $ which implies that
$
(x_{0} \cdot \overline{a}, \overline{\overline{a}}x_{1}, x_{2}\overline{\overline{a}}) = (\overline{\overline{x_{0}}a}, \mu^{-1}(\mu\overline{x_{1}} \times a), a\overline{\overline{x_{2}}}),
$
and so $ x_{2}\overline{\overline{a}} = a\overline{\overline{x_{2}}} $. Assume towards a contradiction that $ x_{2} \neq 0 $. Then since $ x_{2} \in F $, $ x_{2} $ is invertible and $ \overline{\overline{x_{2}}} = x_{2} $. Thus the condition $ x_{2}\overline{\overline{a}} = a\overline{\overline{x_{2}}} $ yields $ a = \overline{\overline{a}} $, and so $ a = \frac{1}{3}\operatorname{T}_{A}(a) \in F $ which is a contradiction. Therefore, $ x = (x_{0},0,0)$, $ x_{0}1 \in F $ which shows that $ \operatorname{Nuc}_{l}(J(A,\mu)) = F$.
\\
Let $ (x_{0},x_{1},x_{2}) \in \operatorname{Nuc}_{r}(J(A,\mu)) $. Then $ (0,0,1)[(0,1,0)(x_{0},x_{1},x_{2})] = [(0,0,1)(0,1,0)](x_{0},x_{1},x_{2}) $  implies that
$
(\overline{\overline{x_{0}}},\mu^{-1}(\overline{\mu \overline{x_{1}}}), \overline{\overline{x_{2}}}) = (x_{0},x_{1},x_{2}).
$
 Hence $ x_{0} = \overline{\overline{x_{0}}} $ and $ x_{2} = \overline{\overline{x_{2}}} $. Using the definition of $ \overline{\overline{x_{0}}} $, we find that $ \overline{\overline{x_{0}}} = \frac{1}{4}(\operatorname{T}_{A}(x_{0})+x_{0}) $, so the condition $ x_{0} = \overline{\overline{x_{0}}} $ gives that $ x_{0} = \frac{1}{4}(\operatorname{T}_{A}(x_{0})+x_{0}) $. Thus $ x_{0} = \frac{1}{3}\operatorname{T}_{A}(x_{0}) \in F $. Also, since $ x_{2} = \overline{\overline{x_{2}}} $, we find in a similar way that $ x_{2} = \frac{1}{3}\operatorname{T}_{A}(x_{2}) \in F $. Next, since $ x = (x_{0},x_{1},x_{2}) \in \operatorname{Nuc}_{r}(J(A,\mu)) $, we have that $ (0,1,0)[(0,0,1)(x_{0},x_{1},x_{2})] = [(0,1,0)(0,0,1)](x_{0},x_{1},x_{2}) $. This implies that
$
(\overline{\overline{x_{0}}}, \overline{\overline{x_{1}}}, \mu (\overline{\mu^{-1}\overline{x_{2}}})) = (x_{0},x_{1},x_{2}),
$
and thus $ x_{1} = \overline{\overline{x_{1}}} $. We find in a similar way that $ x_{1} = \frac{1}{3}\operatorname{T}_{A}(x_{1}) \in F $, i.e.
$ \operatorname{Nuc}_{r}(J(A,\mu)) \subseteq \{(x_{0},x_{1},x_{2}) \in J \hspace{0.1cm} | \hspace{0.1cm} x_{0},x_{1},x_{2} \in F \} $.
\\  Let $ x = (x_{0},x_{1},x_{2}) \in \operatorname{Nuc}_{r}(J(A,\mu)) $, and let $ a \in A \setminus F $. Then $ (0,a,0)[(0,0,1)(x_{0},x_{1},x_{2})] = [(0,a,0)(0,0,1)](x_{0},x_{1},x_{2}) $ which implies that
$
(\overline{a\overline{x_{0}}}, \overline{\overline{x_{1}}}a, \mu(a \times \mu^{-1}\overline{x_{2}})) = (\overline{a} \cdot x_{0}, \overline{\overline{a}}x_{1},x_{2}\overline{\overline{a}}),
$
 therefore $ \overline{\overline{a}}x_{1} = \overline{\overline{x_{1}}}a $. Assume towards a contradiction that $ x_{1} \neq 0 $. Then since $ x_{1} \in F $, $ x_{1} $ is invertible and $ \overline{\overline{x_{1}}} = x_{1} $. Thus the condition $ \overline{\overline{a}}x_{1} = \overline{\overline{x_{1}}}a $ yields $ a = \overline{\overline{a}} $, and so $ a = \frac{1}{3}\operatorname{T}_{A}(a) \in F $ which is a contradiction. Next, since $ (x_{0},x_{1},x_{2}) \in \operatorname{Nuc}_{r}(J(A,\mu)) $, we know that $ (0,0,a)[(0,1,0)(x_{0},x_{1},x_{2})] = [(0,0,a)(0,1,0)](x_{0},x_{1},x_{2}) $ which implies that
$
(\overline{\overline{x_{0}}a},\mu^{-1}(a \times \mu \overline{x_{1}}), a\overline{\overline{x_{2}}}) = (\overline{a} \cdot x_{0}, \overline{\overline{a}}x_{1}, x_{2}\overline{\overline{a}}),
$
and so $ x_{2}\overline{\overline{a}} = a\overline{\overline{x_{2}}} $. Assume towards a contradiction that $ x_{2} \neq 0 $. Then since $ x_{2} \in F $, $ x_{2} $ is invertible and $ \overline{\overline{x_{2}}} = x_{2} $. Thus the condition $ x_{2}\overline{\overline{a}} = a\overline{\overline{x_{2}}} $ yields $ a = \overline{\overline{a}} $, and so $ a = \frac{1}{3}\operatorname{T}_{A}(a) \in F $ which is a contradiction. Therefore, $ x = (x_{0},0,0) = x_{0}1 \in F $ which shows the assertion.
\end{proof}

\begin{theorem} 
Let $A$ be a central simple division algebra of degree three and $ \mu \in A^{\times} $.  Then $ \operatorname{Nuc}_{m}(J(A,\mu)) \subseteq  C(A)  $.
In particular, if $ A \neq F $ is a central simple algebra over \(F\), then $ \operatorname{Nuc}_{m}(J(A,\mu)) = F $, and if
 \(A=K\) is a separable cyclic field extension of $F$,  then $ {\rm Nuc}_m(J(A,\mu)) \subset K$.
\end{theorem}

\begin{proof}
Let $ x = (x_{0},x_{1},x_{2}) \in \operatorname{Nuc}_{m}(J) $, and let $ y_{0} \notin  C(A) $. Then there exists  $ z_{0} \in A $ such that $ y_{0}z_{0} \neq z_{0}y_{0} $. Since $ (x_{0},x_{1},x_{2}) \in \operatorname{Nuc}_{m}(J(A,\mu)) $, we know that
$(y_{0},0,0)[(x_{0},x_{1},x_{2})(z_{0},0,0)] = [(y_{0},0,0)(x_{0},x_{1},x_{2})](z_{0},0,0)$
which implies that
\begin{equation*}
(y_{0} \cdot (x_{0} \cdot z_{0}), \overline{y_{0}}(\overline{z_{0}}x_{1}), (x_{2}\overline{z_{0}})\overline{y_{0}}) = ((y_{0} \cdot x_{0}) \cdot z_{0}, \overline{z_{0}}(\overline{y_{0}}x_{1}),(x_{2}\overline{y_{0}})\overline{z_{0}}).
\end{equation*}
Comparing the second and third components yields
\begin{align}
\overline{y_{0}}(\overline{z_{0}}x_{1}) &= \overline{z_{0}}(\overline{y_{0}}x_{1}), \label{middle nucleus 1} \\ (x_{2}\overline{z_{0}})\overline{y_{0}} &= (x_{2}\overline{y_{0}})\overline{z_{0}}. \label{middle nucleus 2}
\end{align}
Now assume towards a contradiction that $ x_{1} \neq 0 $. Since \(A\) is a division algebra, $ x_{1} $ is invertible. Since \(A\) is associative, (\ref{middle nucleus 1}) implies that $ \overline{y_{0}} \hspace{0.1cm} \overline{z_{0}} = \overline{z_{0}} \hspace{0.1cm} \overline{y_{0}} $. By definition, this yields
\begin{equation*}
(T_{A}(y_{0})1 - y_{0})(T_{A}(z_{0})1-z_{0}) = (T_{A}(z_{0})1-z_{0})(T_{A}(y_{0})1-y_{0}).
\end{equation*}
Hence $ y_{0}z_{0} = z_{0}y_{0} $ which is a contradiction. Now in a similar way we assume towards a contradiction that $ x_{2} \neq 0 $. Then since \(A\) is a division algebra, $ x_{2} $ is invertible. Since \(A\) is associative, (\ref{middle nucleus 2}) implies again that $ \overline{y_{0}} \hspace{0.1cm} \overline{z_{0}} = \overline{z_{0}} \hspace{0.1cm} \overline{y_{0}} $. Hence $ y_{0}z_{0} = z_{0}y_{0} $ which is a contradiction.
 Next, since $ x = (x_{0},0,0) \in \operatorname{Nuc}_{m}(J(A,\mu)) $, we also have that
$(0,1,0)[(x_{0},0,0)(0,0,y_{2})] = [(0,1,0)(x_{0},0,0)](0,0,y_{2})$ for each $ y_{2} \in A $.
This implies
$(\overline{y_{2}\overline{x_{0}}},0,0) = (\overline{\overline{x_{0}}y_{2}},0,0),$
and so $ \overline{y_{2}\overline{x_{0}}} = \overline{\overline{x_{0}}y_{2}} $. By definition, this means that
\begin{equation} \label{Trace equation middle}
\frac{1}{2}T_{A}(\overline{y_{2}}x_{0})1-\frac{1}{2}\overline{y_{2}}x_{0} = \frac{1}{2}T_{A}(x_{0}\overline{y_{2}})1-\frac{1}{2}x_{0}\overline{y_{2}}.
\end{equation}
We know that $ T_{A}(\overline{y_{2}}x_{0}) = T_{A}(x_{0}\overline{y_{2}}) $ (see (\ref{Tr(xy) = Tr(yx)})), and so (\ref{Trace equation middle}) gives that $ \overline{y_{2}}x_{0} = x_{0}\overline{y_{2}} $. By using the definition of $ \overline{y_{2}} $, this implies that $ y_{2}x_{0} = x_{0}y_{2} $. Hence $ x_{0} \in C(A) $. Therefore,
$x = (x_{0},0,0) = x_{0}1 \in C(A).$
Since  $ F \subseteq \operatorname{Nuc}_{m}(J(A,\mu)) $ this implies the assertion if $A$ is a central simple division algebra.
\end{proof}

\begin{theorem} \label{division algebra J(A,mu)} (\cite[Chapter IX, Section 12]{Jacobson}\cite[Chapter C.5]{McC})
For $\mu\in F^\times$,
 $ J(A,\mu) $ is a division algebra if and only if
 $ \mu \notin N_{A}(A^\times) $ and \(A\) is a division algebra, if and only if
 $ N $ is anisotropic.
\end{theorem}

 The general situation is much harder to figure out and we were only able to obtain some obvious necessary conditions:

\begin{theorem} \label{generalied division algebra tits}
Let $\mu\in A^\times$.
\\ (i) 
If $ J(A,\mu) $ is a division algebra then $ \mu \notin N_{A}(A^\times)$ and \(A\) is a division algebra.
\\ (ii) Let $A$ be a division algebra over $F$. If 1, $\mu$, $\mu^2$ are linearly independent over $F$ then $N$ is anisotropic.
\\ (iii) 
If $N$ is anisotropic then $A$ is a division algebra and $ \mu \notin N_{A}(A^\times)$.
\\ (iv) 
Let $ 0\neq x = (x_{0},x_{1},x_{2}) \in J(A,\mu) $.   Then $ x^{\sharp } = 0 $ implies that $A$ has zero divisors, or $A$ is a division algebra and $ \mu \in N_{A}(A^\times) $.
\end{theorem}

\begin{proof}
(i) Suppose that $ J(A,\mu) $ is a division algebra, then so is $A^+$ and thus $A$ (Lemma \ref{A+ division algebra then A}). Assume towards a contradiction that $ \mu = N_{A}(x_0)1 $ for some $ x_0 \in A^\times $. Then $ \mu \in F^\times$ and  $ J(A,\mu) $ is not a division algebra by Theorem \ref{division algebra J(A,mu)}.
 Hence $ \mu \notin N_{A}(A)1 $.
\\ (ii) Since $A$ is a division algebra, $N_A$ is anisotropic. So let $N((x_0,x_1,x_2))=0$, then the assumption means that $N(x_0)=0$, which implies that $x_0=0$. This immediately means that $x_1=x_2=0$, too.
\\ (iii) If $N$ is anisotropic then so is $N_A$, so clearly $A$ is a division algebra. Moreover then $ \mu \notin N_{A}(A^\times)$ by Theorem \ref{division algebra J(A,mu)}.
\\ (iv) Let $ 0\neq x = (x_{0},x_{1},x_{2}) \in J(A,\mu) $.   Then $ x^{\sharp } = 0 $ implies that
\begin{align}
x_{0}^{\sharp } &= x_{1}x_{2} \label{x0 hash} \\ \mu^{-1}x_{2}^{\sharp } &= x_{0}x_{1} \label{x2 hash} \\ \mu x_{1}^{\sharp } &= x_{2}x_{0}. \label{x1 hash}
\end{align}
We can now multiply (\ref{x0 hash}) (resp. (\ref{x2 hash}), (\ref{x1 hash})) by $ x_{0} $ (resp. $ x_{2} $, $ x_{1} $) on the right and left to obtain two new equations. Additionally, using the fact that $ N_{A}(x_{i}) = x_{i}x_{i}^{\sharp } = x_{i}^{\sharp }x_{i} $ for all $ i = 0,1,2 $, we get the following six equations:
\begin{equation} \label{six equations fun}
\begin{aligned}[c]
N_{A}(x_{0}) &= x_{1}x_{2}x_{0} \\
\mu^{-1}N_{A}(x_{2}) &= x_{0}x_{1}x_{2} \\
\mu N_{A}(x_{1}) &= x_{2}x_{0}x_{1}
\end{aligned}
\qquad\qquad
\begin{aligned}[c]
N_{A}(x_{0}) &= x_{0}x_{1}x_{2} \\
\mu^{-1}N_{A}(x_{2}) &= x_{2}x_{0}x_{1} \\
\mu N_{A}(x_{1}) &= x_{1}x_{2}x_{0}.
\end{aligned}
\end{equation}
These imply that $N_{A}(x_{0})=\mu^{-1}N_{A}(x_{2})=\mu N_{A}(x_{1})$. This means that either $N_{A}(x_{0})=\mu^{-1}N_{A}(x_{2})=\mu N_{A}(x_{1})=0$ and so $N_A$ is isotropic, or $N_{A}(x_{0})=\mu^{-1}N_{A}(x_{2})=\mu N_{A}(x_{1})\not=0$ and $N_A$ is anisotropic. In the later case,
$x_{0},x_{1},x_{2}$ are all invertible in $A$,
 $ N_{A}(x_{i}) \neq 0 $ for all $ i = 0,1,2 $ and it follows that $ \mu \in N_{A}(A^\times) $.
 This proves the assertion.
\end{proof}

In other words: If $A$ is a division algebra  and $ \mu \not\in N_{A}(A^\times) $, $ 0\neq x = (x_{0},x_{1},x_{2}) \in J(A,\mu) $, then $ x^{\sharp } \not= 0 $. Note that (iv) was a substantial part of the classical result that if $ \mu \in F^\times$, $ \mu \notin N_{A}(A^\times) $ and \(A\) is a division algebra, then $ N $ is anisotropic. What is missing in order to generalize this result to the generalized first Tits construction is the adjoint identity $ (x^{\sharp })^{\sharp } = N(x)x $. This identity only holds in very special cases, see Lemma \ref{adjoint in J(A,mu)} below. It would be of course desirable to have conditions on when (or if at all) $J(A,\mu)$ is a division algebra.

\subsection{The algebras $J_{lr}(A, \mu)$,  $ J_{rl}(A,\mu)$ and  $ J_{rr}(A,\mu)$}

If \(A\) is a (noncommutative) central simple algebra over $F$  and $ \mu \in A^{\times} $, we can change the positions of $ \mu $  and $ \mu^{-1} $ in the multiplication of $J(A,\mu)$. We now obtain unital nonassociative \(F\)-algebras  $ J_{lr}(A,\mu) $, $ J_{rl}(A,\mu) $ and $
J_{rr}(A,\mu) $ on $A\oplus A\oplus A$, with multiplication given by
$$(x_{0},x_{1},x_{2})(y_{0},y_{1},y_{2}) = (x_{0}\cdot y_{0}+\overline{x_{1}y_{2}}+\overline{y_{1}x_{2}},  \overline{x_{0}}y_{1}+\overline{y_{0}}x_{1}+\mu^{-1}(x_{2}\times y_{2}),
 x_{2}\overline{y_{0}}+y_{2}\overline{x_{0}}+(x_{1}\times y_{1})\mu), $$
$$(x_{0},x_{1},x_{2})(y_{0},y_{1},y_{2}) = (x_{0}\cdot y_{0}+\overline{x_{1}y_{2}}+\overline{y_{1}x_{2}}, \overline{x_{0}}y_{1}+\overline{y_{0}}x_{1}+(x_{2}\times y_{2})\mu^{-1},
x_{2}\overline{y_{0}}+y_{2}\overline{x_{0}}+\mu(x_{1}\times y_{1})),$$
$$(x_{0},x_{1},x_{2})(y_{0},y_{1},y_{2}) = (x_{0}\cdot y_{0}+\overline{x_{1}y_{2}}+\overline{y_{1}x_{2}}, \overline{x_{0}}y_{1}+\overline{y_{0}}x_{1}+(x_{2}\times y_{2})\mu^{-1},
 x_{2}\overline{y_{0}}+y_{2}\overline{x_{0}}+(x_{1}\times y_{1})\mu),$$
respectively. The unital nonassociative algebras $ J_{lr}(A,\mu) $, $ J_{rl}(A,\mu) $ and $ J_{rr}(A,\mu) $  over \(F\)  also generalize an Albert algebra. For $ x = (x_{0},x_{1},x_{2}) \in A \oplus A \oplus A $, we define the \emph{adjoint} on  $ J_{lr}(A,\mu), J_{rl}(A,\mu) $ and $ J_{rr}(A,\mu) $ by
\begin{align*}
x^{\sharp } &= (x_{0}^{\sharp }-x_{1}x_{2},\mu^{-1}x_{2}^{\sharp }-x_{0}x_{1}, x_{1}^{\sharp }\mu-x_{2}x_{0}), \\
x^{\sharp } &= (x_{0}^{\sharp }-x_{1}x_{2},x_{2}^{\sharp }\mu^{-1}-x_{0}x_{1},\mu x_{1}^{\sharp }-x_{2}x_{0}), \\
x^{\sharp } &= (x_{0}^{\sharp }-x_{1}x_{2},x_{2}^{\sharp }\mu^{-1}-x_{0}x_{1}, x_{1}^{\sharp }\mu-x_{2}x_{0}),
\end{align*}
respectively. Let $ y = (y_{0},y_{1},y_{2}) \in A \oplus A \oplus A $. For the \(F\)-algebra $ J = J_{lr}(A,\mu) $, resp. $ J= J_{rl}(A,\mu) $, $ J = J_{rr}(A,\mu) $, we can define again a \emph{(generalized) cubic norm} $ N: J \rightarrow A $, a \emph{(generalized) trace} $ T: J \rightarrow F $, and a \emph{(generalized) sharp map} $ \sharp  :J\times J\rightarrow J$ as
\begin{align*}
N(x) &= N_{A}(x_{0})1+\mu N_{A}(x_{1})+ \mu^{-1}N_{A}(x_{2})-T_{A}(x_{0}x_{1}x_{2})1, \\
T(x) &= T_{A}(x_{0}), \\
x \sharp  y &= (x+y)^{\sharp } - x^{\sharp } - y^{\sharp }.
\end{align*}
It can be verified by a direct computation that
\begin{align*}
x\sharp y &= (x_{0}\sharp y_{0}-x_{1}y_{2}-y_{1}x_{2}, \mu^{-1}(x_{2}\sharp y_{2})-x_{0}y_{1}-y_{0}x_{1},(x_{1}\sharp y_{1})\mu-x_{2}y_{0}-y_{2}x_{0}), \\
x\sharp y &= (x_{0}\sharp y_{0}-x_{1}y_{2}-y_{1}x_{2}, (x_{2}\sharp y_{2})\mu^{-1}-x_{0}y_{1}-y_{0}x_{1},\mu(x_{1}\sharp y_{1})-x_{2}y_{0}-y_{2}x_{0}), \\
x\sharp y &= (x_{0}\sharp y_{0}-x_{1}y_{2}-y_{1}x_{2}, (x_{2}\sharp y_{2})\mu^{-1}-x_{0}y_{1}-y_{0}x_{1},(x_{1}\sharp y_{1})\mu-x_{2}y_{0}-y_{2}x_{0}),
\end{align*}
holds in  $ J_{lr}(A,\mu) $, $ J_{rl}(A,\mu) $ and $ J_{rr}(A,\mu) $, respectively. We will not pursue this approach any further in this paper.

%%%%%%%%%%%%%%%%%%%%%%%%%%%%%%%%%%%%%%%%%%%%%%%%%%%%%%%%%%%%%%%%%%%%%%%%%%%%%%%%%%%%%%%%%%
%
\section{Some more identities} \label{sec:identities}
%
%%%%%%%%%%%%%%%%%%%%%%%%%%%%%%%%%%%%%%%%%%%%%%%%%%%%%%%%%%%%%%%%%%%%%%%%%%%%%%%%%%%%%%%%

\begin{lemma}
Let $ x = (x_{0},x_{1},x_{2}) $, $ y = (y_{0},y_{1},y_{2}), z = (z_{0},z_{1},z_{2}) \in J(A,\mu) $ be such that one of $ x_{1}, y_{1}, z_{1} $ is equal to zero and one of $ x_{2}, y_{2}, z_{2} $ is equal to zero. Then $ T(x \times y, z) = T(x, y \times z) $.
\end{lemma}

\begin{proof}
We find that
\begin{align}
T(x \times y,z) = &\frac{1}{2}T_{A}((x_{0} \sharp y_{0})z_{0}-x_{1}y_{2}z_{0}-y_{1}x_{2}z_{0}) \nonumber \\ &+ \frac{1}{2}T_{A}(\mu^{-1}(x_{2}\sharp y_{2})z_{2}-x_{0}y_{1}z_{2}-y_{0}x_{1}z_{2}) \label{T(x times y,z) = T(x,y times z)} \\ &+\frac{1}{2}T_{A}(\mu(x_{1}\sharp y_{1})z_{1}-x_{2}y_{0}z_{1}-y_{2}x_{0}z_{1}) \nonumber
\end{align}
and
\begin{align}
T(x,y \times z) = &\frac{1}{2}T_{A}(x_{0}(y_{0}\sharp z_{0})-x_{0}y_{1}z_{2}-x_{0}z_{1}y_{2}) \nonumber \\ &+\frac{1}{2}T_{A}(x_{1}\mu(y_{1} \sharp z_{1})-x_{1}y_{2}z_{0}-x_{1}z_{2}y_{0}) \label{T(x times y,z) = T(x,y times z) 2} \\ &+ \frac{1}{2}T_{A}(x_{2}\mu^{-1}(y_{2}\sharp z_{2})-x_{2}y_{0}z_{1}-x_{2}z_{0}y_{1}). \nonumber
\end{align}
Using the definitions, we can show that $ T_{A}((x_{0}\sharp y_{0})z_{0}) = T_{A}(x_{0}(y_{0}\sharp z_{0})) $. Also, since one of  $ x_{1}, y_{1}, z_{1} $ is equal to zero, we have that $ T_{A}(\mu(x_{1}\sharp y_{1})z_{1}) = 0 = T_{A}(x_{1}\mu(y_{1}\sharp z_{1})) $. Finally, since one of  $ x_{2}, y_{2}, z_{2} $ is equal to zero, $ T_{A}(\mu^{-1}(x_{2}\sharp y_{2})z_{2}) = 0 = T_{A}(x_{2}\mu^{-1}(y_{2}\sharp z_{2})) $. Therefore applying these equalities and using (\ref{Tr(xy) = Tr(yx)}), we deduce that (\ref{T(x times y,z) = T(x,y times z)}) and (\ref{T(x times y,z) = T(x,y times z) 2}) are equal, so $ T(x \times y,z) = T(x,y \times z) $.
\end{proof}

We know that $xx^{\sharp } = x^{\sharp }x = N(x)1$ holds for all $ x \in J(A,\mu) $ if $ \mu \in F^{\times} $.
We now show for which $ x \in J(A,\mu) $ we still obtain $xx^{\sharp } = x^{\sharp }x = N(x)1$:

\begin{lemma} \label{generalization sharp id}
Let $\mu\in A^\times$ and suppose that $ x \in J(A,\mu) $, such that one of the following holds:
\\ (i)  $ x = (x_{0},0,x_{2}) \in J(A,\mu) $, $ x_{0} \mu = \mu x_{0} $ and $ N_{A}(x_{2}) = 0 $.
\\ (ii)  $ x = (x_{0},x_{1},0) \in J(A,\mu) $, $ x_{1}\mu = \mu x_{1} $ and $ N_{A}(x_{1}) = 0 $.
\\ Then we have
\begin{equation} \label{xxhas = xhashx = N(x)}
xx^{\sharp } = x^{\sharp }x = N(x)1.
\end{equation}
Moreover, assume that one of the following holds:
\\ (iii) $ x = (x_{0},0,x_{2}) \in J(A,\mu) $, $ x_{0} \mu = \mu x_{0} $ and $ N_{A}(x_{2}) \neq 0 $.
\\ (iv) $ x = (x_{0},x_{1},0) \in J(A,\mu) $, $ x_{1}\mu = \mu x_{1} $ and $ N_{A}(x_{1}) \neq 0 $.
\\ Then $ xx^{\#} = x^{\#}x = N(x)1 $ if and only if $ \mu \in F^{\times}$.
\end{lemma}

\begin{proof}
 Let $ x = (x_{0},x_{1},x_{2})\in J(A,\mu)  $, and let $ xx^{\sharp } = (a_{0},a_{1},a_{2}) $, then
$x^{\sharp } = (x_{0}^{\sharp }-x_{1}x_{2},\mu^{-1}x_{2}^{\sharp }-x_{0}x_{1},\mu x_{1}^{\sharp }-x_{2}x_{0}).$
Thus we have
\begin{align}
a_{0} &= \frac{1}{2}(x_{0}x_{0}^{\sharp }-x_{0}x_{1}x_{2}+x_{0}^{\sharp }x_{0}-x_{1}x_{2}x_{0}) \nonumber \\ &\hspace{0.3cm} +\frac{1}{2}(T_{A}(x_{1}\mu x_{1}^{\sharp }-x_{1}x_{2}x_{0})1-x_{1}\mu x_{1}^{\sharp }+x_{1}x_{2}x_{0}) \label{a0 equation} \\ &\hspace{0.3cm} +\frac{1}{2}(T_{A}(\mu^{-1}x_{2}^{\sharp }x_{2}-x_{0}x_{1}x_{2})1-\mu^{-1}x_{2}^{\sharp }x_{2}+x_{0}x_{1}x_{2}), \nonumber
\end{align}
\begin{align}
a_{1} &= \frac{1}{2}(T_{A}(x_{0})-x_{0})(\mu^{-1}x_{2}^{\sharp }-x_{0}x_{1})+\frac{1}{2}(T_{A}({x_{0}^{\sharp }-x_{1}x_{2}})-x_{0}^{\sharp }+x_{1}x_{2})x_{1} \nonumber \\ & \hspace{0.3cm} + \mu^{-1}(x_{2} \times (\mu x_{1}^{\sharp }-x_{2}x_{0})) \label{a1 equation},
\end{align}
\begin{align}
a_{2} &= \frac{1}{2}x_{2}(T_{A}(x_{0}^{\sharp}-x_{1}x_{2})-x_{0}^{\sharp}+x_{1}x_{2}) +\frac{1}{2}(\mu x_{1}^{\sharp}-x_{2}x_{0})(T_{A}(x_{0})-x_{0}) \nonumber \\  & \hspace{0.3cm} + \mu(x_{1} \times (\mu^{-1}x_{2}^{\sharp}-x_{0}x_{1})) \label{a2 equation}
\end{align}
by the definition of the multiplication on $ J(A,\mu) $.
\\ (i) If $ x_{1} = 0 $ and $ N_{A}(x_{2}) = 0 $, using the fact that $ x_{i}x_{i}^{\sharp } = x_{i}^{\sharp }x_{i} = N_{A}(x_{i})1 $ for all $ i = 0,1,2 $ (see (\ref{degree three  identity equivalent})), (\ref{a0 equation}) simplifies to  $ a_{0} = N_{A}(x_{0})1 = N(x) $. Since we have $ x_{0} \mu^{-1} = \mu^{-1}x_{0} $,  (\ref{a1 equation}) gives that
$
a_{1} = \frac{1}{2}\mu^{-1}(\operatorname{T}_{A}(x_{0})-x_{0})x_{2}^{\sharp} - \mu^{-1}(x_{2} \times (x_{2}x_{0})). \label{a1 equation2}
$
By Lemma \ref{sharp alternative},
\begin{equation} \label{equation that I need}
x_{2}^{\sharp }\sharp  x_{0} = x_{2}^{\sharp }x_{0}+x_{0}x_{2}^{\sharp }-T_{A}(x_{2}^{\sharp })x_{0}-T_{A}(x_{0})x_{2}^{\sharp }+(T_{A}(x_{2}^{\sharp })T_{A}(x_{0})-T_{A}(x_{2}^{\sharp }x_{0}))1.
\end{equation}
Using the fact that $ T_{A}(x_{2}^{\sharp }) = S_{A}(x_{2}) $ (by (\ref{use it twice})) on the right-hand side of (\ref{equation that I need}), we further obtain after some simplification that
\begin{equation} \label{i will need this}
x_{2}^{\sharp }\sharp  x_{0} = x_{2}^{2}x_{0}-T_{A}(x_{2})x_{2}x_{0}+x_{0}x_{2}^{\sharp }-T_{A}(x_{0})x_{2}^{\sharp }-T_{A}(x_{2}^{2}x_{0})1+T_{A}(x_{2}x_{0})1.
\end{equation}
Now combining (\ref{McCrimmon xyx equation}) 
 with (\ref{i will need this}) yields
$
T_{A}(x_{0})x_{2}^{\sharp }-x_{0}x_{2}^{\sharp } =
x_{2}^{2}x_{0}+x_{2}x_{0}x_{2}-T_{A}(x_{2})x_{2}x_{0}-T_{A}(x_{2}x_{0})x_{2}+(T_{A}(x_{2})T_{A}(x_{2}x_{0})-T_{A}(x_{2}^{2}x_{0}))1
=2(x_{2} \times (x_{2}x_{0})) $, so $ x_{2} \times (x_{2}x_{0}) = \frac{1}{2}(T_{A}(x_{0})x_{2}^{\sharp}-x_{0}x_{2}^{\sharp }) $. Hence (\ref{a1 equation2}) implies $ a_{1} = 0 $. For $ a_{2} $, (\ref{a2 equation}) yields
$
a_{2} = \frac{1}{2}x_{2}(T_{A}(x_{0}^{\sharp})-x_{0}^{\sharp})-\frac{1}{2}x_{2}(T_{A}(x_{0})x_{0}-x_{0}^{2}).
$
Then using the definition of $ x_{0}^{\sharp} $ and the fact that $ 2S_{A}(x_{0}) = \operatorname{T}_{A}(x_{0})^{2}-\operatorname{T}_{A}(x_{0}^{2}) $, we find that $ \operatorname{T}_{A}(x_{0}^{\sharp})-x_{0}^{\sharp} = \operatorname{T}_{A}(x_{0})x_{0}-x_{0}^{2} $. Therefore, $ a_{2} = 0 $.
\\ (ii) In this case we have $ x_{2} = 0 $, $ x_{1}\mu = \mu x_{1} $ and $ N_{A}(x_{1}) = 0 $. So  (\ref{a0 equation}) simplifies to $ a_{0} = N_{A}(x_{0})1 = N(x) $. For $ a_{1} $, (\ref{a1 equation}) simplifies to
$
a_{1} = -\frac{1}{2}(\operatorname{T}_{A}(x_{0})x_{0}-x_{0}^{2})x_{1}+\frac{1}{2}(\operatorname{T}_{A}(x_{0}^{\sharp})-x_{0}^{\sharp})x_{1}.
$
Then in a similar way to how we found $ a_{2} $ in (i), we find here that $ a_{1} = 0 $. For $ a_{2} $, (\ref{a2 equation}) simplifies to
$
a_{2} = \frac{1}{2}\mu x_{1}^{\sharp}(T_{A}(x_{0})-x_{0}) - \mu(x_{1} \times (x_{0}x_{1})).
$
We now find in a similar way to how we found $ a_{1} $ in (i) that $ a_{2} = 0 $.
\\
 To prove that the claimed equivalence holds assuming (iii) or (iv), we only need to show the forward direction since we know from the classical first Tits construction that the reverse direction holds:
\\
(iii) Here, (\ref{a0 equation}) yields
$
a_{0} = N_{A}(x_{0})1+\frac{1}{2}(\operatorname{T}_{A}(\mu^{-1})N_{A}(x_{2})-\mu^{-1}N_{A}(x_{2})),
$
thus  $ xx^{\sharp} = N(x)1 = (N_{A}(x_{0})1+\mu^{-1}N_{A}(x_{2}))1 $ gives that
$
N_{A}(x_{0})1+\frac{1}{2}(\operatorname{T}_{A}(\mu^{-1})N_{A}(x_{2})-\mu^{-1}N_{A}(x_{2})) = a_{0} = N_{A}(x_{0})1+\mu^{-1}N_{A}(x_{2}).
$
Therefore we have $ \mu^{-1} = \frac{1}{3} \operatorname{T}_{A}(\mu^{-1}) \in F^{\times} $, so $ \mu \in F^{\times} $.
\\ (iv) In this case, (\ref{a0 equation}) yields
$
a_{0} = N_{A}(x_{0})1 + \frac{1}{2}(\operatorname{T}_{A}(\mu)N_{A}(x_{1})-\mu N_{A}(x_{1})),
$
thus  $ xx^{\sharp} = N(x)1 = (N_{A}(x_{0})1+\mu N_{A}(x_{1}))1 $ yields
$
N_{A}(x_{0})1 + \frac{1}{2}(\operatorname{T}_{A}(\mu)N_{A}(x_{1})-\mu N_{A}(x_{1})) = a_{0} = N_{A}(x_{0})1+\mu N_{A}(x_{1}).
$
Therefore we obtain $ \mu = \frac{1}{3}\operatorname{T}_{A}(\mu) \in F^{\times} $.
The proof that $ x^{\sharp }x = N(x)1 $ is done similarly.
\end{proof}

\begin{corollary}
Let $\mu\in A^\times$. Suppose that $ x \in J(A,\mu) $ satisfies $N(x)\not=0$, and assume that one of the following holds:
\\ (i)  $ x = (x_{0},x_{1},0) \in J(A,\mu) $, $ x_{1}\mu = \mu x_{1} $ and $ N_{A}(x_{1}) = 0 $.
\\ (ii)  $ x = (x_{0},0,x_{2}) \in J(A,\mu) $, $ x_{0} \mu = \mu x_{0} $ and $ N_{A}(x_{2}) = 0 $.
\\ Then $x$ is invertible in $J(A,\mu) $ with $x^{-1}=  N(x)^{-1}x^{\sharp } $.
\end{corollary}

\begin{proof}
Let $\mu\in A^\times$ and suppose that $ x \in J(A,\mu) $ satisfies (i) or (ii), then $xx^{\sharp } = x^{\sharp }x = N(x)1.$
Since $F=C(J(A,\mu))$ this yields the assertion.
\end{proof}

In particular, if $N$ is anisotropic, then every $0 \neq x  x = (x_{0},x_{1},0) \in J(A,\mu) $ in (i) or (ii) is of the type
$ x = (x_{0},0,0) \in J(A,\mu) $, i.e. lies in $A$, so this result becomes trivial then.

\begin{corollary}
Let $\mu\in A^\times$ and suppose that $ x \in J(A,\mu) $, such that one of the following holds:
\\ (i)  $ x = (x_{0},0,x_{2}) \in J(A,\mu) $, $ x_{0} \mu = \mu x_{0} $ and $ N_{A}(x_{2}) = 0 $.
\\ (ii)  $ x = (x_{0},x_{1},0) \in J(A,\mu) $, $ x_{1}\mu = \mu x_{1} $ and $ N_{A}(x_{1}) = 0 $.
\\ Then we have
\begin{equation*}
x^{3} - T(x)x^{2}+S(x)x-N(x)1 = 0.
\end{equation*}
\end{corollary}

\begin{proof}
Using the fact that $ x^{\sharp} = x^{2}-T(x)x+S(x)1 $ from Lemma \ref{some relations} (i), we have that $ x^{3} - T(x)x^{2}+S(x)x-N(x)1 = 0 $ if and only if $ xx^{\sharp} = x^{\sharp}x = N(x)1 $. Thus the result now follows as a consequence of Lemma \ref{generalization sharp id}.
\end{proof}

\begin{theorem}
The identity $ xx^{\sharp } = x^{\sharp }x = N(x)1 $  
 holds for all $ x \in J(A,\mu) $ if and only if $ \mu \in F^{\times} $.
\end{theorem}

\begin{proof}
If $ \mu \in F^{\times} $, then
$xx^{\sharp } = x^{\sharp }x = (N(x),0,0)$ for all $ x \in J(A,\mu) $. 
Conversely, suppose that $ xx^{\sharp } = x^{\sharp }x = (N(x),0,0) $ holds for all $ x \in J(A,\mu) $. Take $ x = (0,1,0) $. Then $ x^{\sharp } = (0,0,\mu) $, and so
\begin{equation*}
xx^{\sharp } = (\bar{\mu},0,0) = (\frac{1}{2}(T_{A}(\mu)1-\mu),0,0).
\end{equation*}
We also know that by definition, $ N(x) = \mu N_{A}(1) = \mu $, so the condition $ xx^{\sharp } = (N(x),0,0) $ gives that $ \mu = \frac{1}{2}(T_{A}(\mu)1-\mu) $. Hence
$\mu = \frac{1}{3}T_{A}(\mu)1 \in F^{\times}.$
\end{proof}

A similar argument also shows that $ xx^{\sharp } = x^{\sharp }x = N(x)1 $ only holds for $x\in A$.

\ignore{%%%
your argument below only shows:
If $A=K$ is a cyclic field extension of degree three  and if  $(x^\sharp)^\sharp=N(x)$
holds for all $ x  \in J(A,\mu) $, then it follows that  $ \mu \in C(A)=K $. So $ \mu \in K^{\times} $, can we make stronger?
}%%%%

We know that the adjoint identity $ (x^{\#})^{\#} = N(x)x$ 
holds for all $ x \in J(A,\mu) $, if $ \mu \in F^{\times} $ \cite[Chapter C.4]{McC}. In the general construction, it holds only in very special cases:

\begin{lemma}  \label{adjoint in J(A,mu)}
Let $\mu\in A^\times$ and suppose that $ x \in J(A,\mu) $, such that one of the following holds:
\\ (i)  $ x = (0,x_{1},0) \in J(A,\mu) $ and $ N_{A}(x_{1}) = 0 $.
\\ (ii)  $ x = (x_{0},x_{1},0) \in J(A,\mu) $ and $ N_{A}(x_{1}) = 0 $ and $ x_{1}\mu = \mu x_{1} $.
\\ (iii)  $ x = (x_{0},0,x_{2}) \in J(A,\mu) $ and $ N_{A}(x_{2}) = 0 $ and $ x_{0}\mu = \mu x_{0} $.
\\ Then we have $(x^{\#})^{\#} = N(x)x$. 
\\ Moreover, if one of the following holds:
 \\ (iv) $ x = (x_{0},x_{1},0) \in J(A,\mu) $, $ N_{A}(x_{1}) \neq 0 $, $ x_{0}\mu = \mu x_{0} $ and $ x_{1}\mu = \mu x_{1} $.
 \\ (v) $ x = (x_{0},0,x_{2}) \in J(A,\mu) $, $ N_{A}(x_{2}) \neq 0 $, $ x_{0}\mu = \mu x_{0} $ and $ x_{2}\mu = \mu x_{2} $.
 \\ Then $(x^{\#})^{\#} = N(x)x$ for all $x\in J(A,\mu)$ if and only if $ N_{A}(\mu) = \mu^{3} $.
\end{lemma}

\begin{proof}
Let $ x = (x_{0},x_{1},x_{2}) \in J(A,\mu) $ and $ (x^{\sharp })^{\sharp } = (a_{0},a_{1},a_{2}) $. By definition,
$
x^{\sharp } = (x_{0}^{\sharp }-x_{1}x_{2},\mu^{-1}x_{2}^{\sharp }-x_{0}x_{1},\mu x_{1}^{\sharp }-x_{2}x_{0}),
$
so $ a_{0} = (x_{0}^{\sharp }-x_{1}x_{2})^{\sharp }-(\mu^{-1} x_{2}^{\sharp }-x_{0}x_{1})(\mu x_{1}^{\sharp }-x_{2}x_{0}) $. Now using (\ref{S(x) = Tr - Tr}) and Lemma \ref{sharp alternative}, it is easy to show that
\begin{align*}
(x_{0}^{\sharp }-x_{1}x_{2})^{\sharp } &= (x_{0}^{\sharp }-x_{1}x_{2})^{2} - T_{A}(x_{0}^{\sharp }-x_{1}x_{2})(x_{0}^{\sharp }-x_{1}x_{2})+S_{A}(x_{0}^{\sharp }-x_{1}x_{2}) \\ &= (x_{0}^{\sharp })^{\sharp }-x_{0}^{\sharp }\sharp  (x_{1}x_{2})+(x_{1}x_{2})^{\sharp }.
\end{align*}
Hence
\begin{align}
a_{0} &= (x_{0}^{\sharp }-x_{1}x_{2})^{\sharp }-(\mu^{-1} x_{2}^{\sharp }-x_{0}x_{1})(\mu x_{1}^{\sharp }-x_{2}x_{0}) \nonumber \\ &= (x_{0}^{\sharp })^{\sharp }-x_{0}^{\sharp }\sharp  (x_{1}x_{2})+ (x_{1}x_{2})^{\sharp } - \mu^{-1}x_{2}^{\sharp}\mu x_{1}^{\sharp }+\mu^{-1}x_{2}^{\sharp }x_{2}x_{0}+x_{0}x_{1}\mu x_{1}^{\sharp } - x_{0}x_{1}x_{2}x_{0}. \label{a0 equation adjoint}
\end{align}
Similarly, we find that
\begin{align}
a_{1} &= \mu^{-1}(\mu x_{1}^{\sharp}-x_{2}x_{0})^{\sharp}-(x_{0}^{\sharp}-x_{1}x_{2})(\mu^{-1}x_{2}^{\sharp}-x_{0}x_{1}) \nonumber \\ &= \mu^{-1}((\mu x_{1}^{\sharp})^{\sharp}-(\mu x_{1}^{\sharp})\#(x_{2}x_{0})+(x_{2}x_{0})^{\sharp}) \nonumber \\ & \hspace{0.3cm} -x_{0}^{\sharp}\mu^{-1}x_{2}^{\sharp} +x_{0}^{\sharp}x_{0}x_{1}+x_{1}x_{2}\mu^{-1}x_{2}^{\sharp}-x_{1}x_{2}x_{0}x_{1} \label{a1 equation adjoint}
\end{align}
and
\begin{align}
a_{2} &= \mu(\mu^{-1} x_{2}^{\sharp}-x_{0}x_{1})^{\sharp}-(\mu x_{1}^{\sharp}-x_{2}x_{0})(x_{0}^{\sharp}-x_{1}x_{2}) \nonumber \\ &= \mu((\mu^{-1} x_{2}^{\sharp})^{\sharp}-(\mu^{-1} x_{2}^{\sharp})\sharp (x_{0}x_{1})+(x_{0}x_{1})^{\sharp}) \nonumber \\ & \hspace{0.3cm} -\mu x_{1}^{\sharp}x_{0}^{\sharp} + \mu x_{1}^{\sharp} x_{1}x_{2}+x_{2}x_{0}x_{0}^{\sharp}-x_{2}x_{0}x_{1}x_{2}. \label{a2 equation adjoint}
\end{align}
 (i) Here $ x_{0} = x_{2} = 0 $, therefore (\ref{a0 equation adjoint}) implies $ a_{0} = 0 $ 
 and (\ref{a1 equation adjoint}) gives that
$
a_{1} = \mu^{-1}(\mu x_{1}^{\sharp})^{\sharp} = \mu^{-1}N_{A}(x_{1})x_{1} \mu^{\sharp} = 0 = N(x)x_{1}.
$
Finally, (\ref{a2 equation adjoint}) gives that $ a_{2} = 0 $ as required.
\\ (ii) Since $ x_{2} = 0 $, $ x_{1}\mu = \mu x_{1} $ and $ N_{A}(x_{1}) = 0 $, we find by (\ref{a0 equation adjoint}) that
$
a_{0} = (x_{0}^{\sharp})^{\sharp} +x_{0}\mu x_{1}x_{1}^{\sharp} = N_{A}(x_{0})x_{0}+x_{0}\mu N_{A}(x_{1}) = N(x)x_{0}.
$
Now (\ref{a1 equation adjoint}) gives that
$
a_{1} = \mu^{-1}(\mu x_{1}^{\sharp})^{\sharp} + x_{0}^{\sharp}x_{0}x_{1} = \mu^{-1}N_{A}(x_{1})x_{1}\mu^{\sharp} + N_{A}(x_{0})x_{1} = N(x)x_{1},
$
and by (\ref{a2 equation adjoint}) we obtain
$
a_{2} = \mu(x_{0}x_{1})^{\sharp} - \mu x_{1}^{\sharp}x_{0}^{\sharp} = 0  = N(x)0.
$
\\ (iii) Since $ x_{1} = 0 $ and $ N_{A}(x_{2}) = 0 $, (\ref{a0 equation adjoint}) yields
$
a_{0} = (x_{0}^{\sharp})^{\sharp}+\mu^{-1}x_{2}^{\sharp}x_{2}x_{0} = N_{A}(x_{0})x_{0}+\mu^{-1}N_{A}(x_{2})x_{0} = N(x)x_{0}.$
Now since $ x_{0}\mu^{-1} = \mu^{-1}x_{0} $, we have that $ x_{0}^{\sharp}\mu^{-1} = \mu^{-1}x_{0}^{\sharp} $, so (\ref{a1 equation adjoint}) gives that
$
a_{1} = \mu^{-1}(x_{2}x_{0})^{\sharp} - \mu^{-1}x_{0}^{\sharp}x_{2}^{\sharp} = 0 = N(x)0.
$
Finally, (\ref{a2 equation adjoint}) gives that
$
a_{2} = \mu(\mu^{-1}x_{2}^{\sharp})^{\sharp}+x_{2}x_{0}x_{0}^{\sharp} = \mu N_{A}(x_{2})x_{2}(\mu^{-1})^{\sharp}+N_{A}(x_{0})x_{2} = N(x)x_{2}.$
\\
 (iv) Since $ x_{2} = 0 $, $ x_{0}\mu = \mu x_{0} $ and $ x_{1}\mu = \mu x_{1} $, (\ref{a0 equation adjoint}) yields
$
a_{0} = (x_{0}^{\sharp})^{\sharp}+\mu x_{0}x_{1}x_{1}^{\sharp} = N_{A}(x_{0})x_{0}+\mu N_{A}(x_{1})x_{0} = N(x)x_{0}.
$
Now (\ref{a2 equation adjoint}) gives that
$a_{2} = \mu (x_{0}x_{1})^{\sharp}-\mu x_{1}^{\sharp}x_{0}^{\sharp} = 0 $. 
Finally, (\ref{a1 equation adjoint}) gives that
$a_{1} = \mu^{-1}(\mu x_{1}^{\sharp})^{\sharp}+x_{0}^{\sharp}x_{0}x_{1} = \mu^{-1}N_{A}(x_{1})x_{1}\mu^{\sharp}+N_{A}(x_{0})x_{1}.$
Thus $ a_{1} = N(x)x_{1} $ if and only if $ \mu^{-1}N_{A}(x_{1})x_{1}\mu^{\sharp}+N_{A}(x_{0})x_{1} = N(x)x_{1} $, which occurs if and only if $ \mu^{-1}N_{A}(x_{1})x_{1}\mu^{\sharp} = \mu N_{A}(x_{1})x_{1} $. Since $ N_{A}(x_{1}) \neq 0 $ and $ x_{1}\mu^{\sharp} = \mu^{\sharp}x_{1} $, this occurs if and only if $ \mu^{\sharp} x_{1} = \mu^{2}x_{1} $. Finally, $ N_{A}(x_{1}) \neq 0 $ implies that $ x_{1} $ is invertible, so $ \mu^{\sharp}x_{1} = \mu^{2}x_{1} $ if and only if $ N_{A}(\mu) = \mu \mu^{\sharp} = \mu^{3} $.
\\ (v) Since $ x_{1} = 0 $, (\ref{a0 equation adjoint}) yields
$
a_{0} = (x_{0}^{\sharp})^{\sharp} +\mu^{-1}x_{2}^{\sharp}x_{2}x_{0} = N(x)x_{0}.
$
Also, since $ x_{0} $ commutes with $ \mu $, $ x_{0}^{\sharp} $ commutes with $ \mu $. So $ x_{0}^{\sharp}\mu^{-1} = \mu^{-1}x_{0}^{\sharp} $. Hence (\ref{a1 equation adjoint}) gives that
$a_{1} = \mu^{-1}(x_{2}x_{0})^{\sharp}-\mu^{-1}x_{0}^{\sharp}x_{2}^{\sharp} = 0 = N(x)0.$
Finally, (\ref{a2 equation adjoint}) yields
$
a_{2} = \mu(\mu^{-1}x_{2}^{\sharp})^{\sharp}+x_{2}x_{0}x_{0}^{\sharp} = \mu N_{A}(x_{2})x_{2}(\mu^{-1})^{\sharp}+N_{A}(x_{0})x_{2}.
$
 Thus $ a_{2} = N(x)x_{2} $ if and only if $ \mu N_{A}(x_{2})x_{2}(\mu^{-1})^{\sharp}+N_{A}(x_{0})x_{2} = N(x)x_{2} $, which occurs if and only if $ \mu N_{A}(x_{2})x_{2}(\mu^{-1})^{\sharp} = \mu^{-1}N_{A}(x_{2})x_{2} $. Since $ N_{A}(x_{2}) \neq 0 $ and $ x_{2}(\mu^{-1})^{\sharp} = (\mu^{-1})^{\sharp}x_{2} $, this occurs if and only if $ (\mu^{-1})^{\sharp}x_{2} = \mu^{-2}x_{2} $. Finally, $ N_{A}(x_{2}) \neq 0 $ implies that $ x_{2} $ is invertible, so $ (\mu^{-1})^{\sharp}x_{2} = \mu^{-2}x_{2} $ if and only if $ N_{A}(\mu^{-1}) = \mu^{-1}(\mu^{-1})^{\sharp} = \mu^{-3} $. This is equivalent to $ N_{A}(\mu) = \mu^{3} $.
\end{proof}

\begin{proposition} \label{adjoint in J(A,muA)}
 Let $A$ be a central simple algebra over $F$. Then  $(x^\sharp)^\sharp=N(x)$ for all $ x  \in J(A,\mu) $ if and only if $ \mu \in F^{\times} $.
\end{proposition}

\begin{proof}
 Let $ \mu \in F^{\times}$ then by Lemma \ref{adjoint in J(A,mu)}, the adjoint identity holds for all $ x \in J(A,\mu) $.
 Suppose now that the adjoint identity holds for all  $x\in J(A,\mu) $. Let $ x = (x_{0},1,0) \in J(A,\mu) $ for some $ x_{0} \in A $. Then $ x^{\sharp } = (x_{0}^{\sharp },-x_{0},\mu) $ and so
\begin{equation} \label{adjoint proof in generalised 1}
(x^{\sharp })^{\sharp } = ((x_{0}^{\sharp })^{\sharp }+x_{0}\mu,\mu^{-1}\mu^{\sharp }+(x_{0}^{\sharp })^{\sharp }x_{0},0).
\end{equation}
Also, $ N(x) = N_{A}(x_{0})1+\mu $. Since the adjoint identity holds by assumption, we see that by using (\ref{adjoint proof in generalised 1}),
\begin{equation} \label{adjoint proof needK}
((x_{0}^{\sharp })^{\sharp }+x_{0}\mu,\mu^{-1}\mu^{\sharp }+(x_{0}^{\sharp })^{\sharp }x_{0},0) = (N_{A}(x_{0})x_{0}+\mu x_{0},N_{A}(x_{0})+\mu,0).
\end{equation}
 We know that $ (x_{0}^{\sharp })^{\sharp } = N_{A}(x_{0})x_{0} $ for all $x_0\in A$ by Lemma \ref{adjoint identity lemma}, and so by comparing the first components of (\ref{adjoint proof needK}), we find that $ x_{0}\mu = \mu x_{0} $ for all $x_0\in A$. Hence $ \mu \in C(A) $, and since \(A\) is a central simple algebra   by assumption, $ \mu \in F^{\times} $.
\end{proof}

 If $\mu\in F^\times$ then the norm $N$ \emph{permits Jordan composition}, i.e.
$N(U_xy) = N_A(x)^2 N(y)$ for all $x,y \in  J(A, a )$. The following result is a corrected version of \cite[Theorem 5.2.5]{St2}, and a weak generalization of the Jordan composition for $\mu\in A^\times \setminus F$:

\begin{theorem}
Let $ x = (x_{0},0,0)\in A $, $ y = (y_{0},y_{1},y_{2}) \in J(A,\mu) $  and
suppose that one of the following holds:
\\ (i) $T_{A}(y_{0}y_{1}y_{2})=T_{A}(N_{A}(y_{0})x_{0}^{\sharp }y_{1}y_{2}x_{0}^{\sharp }).$
\\ (ii) $y_{0}y_{1}y_{2}=N_{A}(y_{0})x_{0}^{\sharp }y_{1}y_{2}x_{0}^{\sharp }$.
\\ (iii)  $ y_{i} = 0 $ for some $ i = 0,1,2 $.
\\  Then $ N(U_{x}(y)) = N(x)^{2}N(y) $.
\end{theorem}

\begin{proof}
Using the definitions, we see that $ T(x,y) = T_{A}(x_{0}y_{0}) $ and $ x^{\sharp }\sharp  y =  (x_{0}^{\sharp }\sharp  y_{0},-x_{0}^{\sharp }y_{1},-y_{2}x_{0}^{\sharp }) $. So
$U_{x}(y) = T(x,y)x-x^{\sharp }\sharp  y = (U_{x_{0}}(y_{0}),x_{0}^{\sharp }y_{1},y_{2}x_{0}^{\sharp }).$
This yields
\begin{align*}
N(U_{x}(y))
&= N_{A}(U_{x_{0}}(y_{0}))1+\mu N_{A}(x_{0}^{\sharp }y_{1})+\mu^{-1}N_{A}(y_{2}x_{0}^{\sharp })-T_{A}(U_{x_{0}}(y_{0})x_{0}^{\sharp }y_{1}y_{2}x_{0}^{\sharp })1 \\
&= N_{A}(x_{0})^{2}(N_{A}(y_{0})1+\mu N_{A}(y_{1})+\mu^{-1}N_{A}(y_{2})-T_{A}(N_{A}(y_{0})x_{0}^{\sharp }y_{1}y_{2}x_{0}^{\sharp })1) \\
&= N(x)^{2}(N(y)+T_{A}(y_{0}y_{1}y_{2})1-T_{A}(N_{A}(y_{0})x_{0}^{\sharp }y_{1}y_{2}x_{0}^{\sharp })1),
\end{align*}
where in the second equality we have used the fact that $ N_{A}(x_{0}^{\sharp }) = N_{A}(x_{0})^{2} $, and that $ N_{A}(U_{x_{0}}(y_{0})) = N_{A}(x_{0})^{2}N_{A}(y_{0}) $. Therefore, $ N(U_{x}(y)) = N(x)^{2}N(y) $, if and only if
$T_{A}(y_{0}y_{1}y_{2})1=T_{A}(N_{A}(y_{0})x_{0}^{\sharp }y_{1}y_{2}x_{0}^{\sharp }).$
\\ (ii) and (iii) are examples where this is the case.
\end{proof}

 \begin{remark} Let $ f: J(A,\mu) \rightarrow J(A,\mu) $ be an automorphism. Then
\begin{equation} \label{automorphism linear}
f((x_{0},x_{1},x_{2})) = f((x_{0},0,0))+f((0,x_{1},0))+f((0,0,x_{2})).
\end{equation}
Now for each $ x \in A $, we have
$f((0,\overline{x},0)) = f((x,0,0))f((0,1,0), $ and $ f((0,0,\overline{x})) = f((x,0,0))f((0,0,1))$.
On the other hand, by using the definition of $ \overline{x} $,
\begin{align*}
f((0,\overline{x},0)) &= \frac{1}{2}\operatorname{Tr}_{A}(x)f((0,1,0))-\frac{1}{2}f((0,x,0)), \\ f((0,0,\overline{x})) &= \frac{1}{2}\operatorname{Tr}_{A}(x)f((0,0,1))-\frac{1}{2}f((0,0,x)).
\end{align*}
Hence
\begin{align}
f((0,x,0)) &= f((0,1,0))(\operatorname{Tr}_{A}(x)-2f((x,0,0))), \label{2 overline} \\ f((0,0,x)) &= f((0,0,1))(\operatorname{Tr}_{A}(x)-2f((x,0,0))). \label{3 overline}
\end{align}
So by (\ref{automorphism linear}), (\ref{2 overline}) and (\ref{3 overline}), we see that any automorphism of $ J(A,\mu) $ is determined by its restriction on $ A^{+} $, and its value on $ (0,1,0) $ and $ (0,0,1) $.
Let $f:J(A,\mu)\rightarrow J(A,\mu)$ be an automorphism that fixes $A^+$ then  $f|_{A^+}=\tau$ is either an automorphism or an anti-automorphism of $A$. Moreover, clearly $f((1,0,0))=(1,0,0)$, so
$$f((x_0,x_1,x_2))=(\tau(x_0),0,0)+(\tau(x_1),0,0)f((0,1,0))+(\tau(x_2),0,0)f((0,0,1)).$$
Calculation to try gain some deeper understanding on the automorphisms are tedious and did not lead us anywhere so far.
\end{remark}

\section{The nine-dimensional nonassociative algebras $J(K,\mu)$} \label{sec:example}

Let $K/F$ be a separable cubic field extension  with ${\rm Gal}(K/F)=\langle \sigma \rangle$,  norm $N_K$, and trace $T_K$.
For all $x_0 \in K$, we have
$x_0^\sharp = \sigma(x_0) \sigma^2(x_0)$ and $\overline{x_0}=\frac{1}{2}(\sigma(x_0)+\sigma^2(x_0))$. Assume $\mu\in K^\times$.

Let us compare the first Tits construction $J(K,\mu)$ with the algebra $D^+$ for a (perhaps nonassociative) cyclic algebra $D=(K/F, \sigma, \mu)$ over $F$ of degree three. Consider $D$ as a left $K$-vector space with basis $\{1,z, z^2\}.$
 Write $R_x$ for the matrix of right multiplication by $x = x_0 + x_1z + x_2z^2,$ $ x_i \in K$, with respect to the basis $\{1, z, z^2\}$, then the cubic map  $N_D: D \rightarrow K$, $N_D(x) = {\rm det}(R_x)$ (which is the reduced norm of the central simple algebra $D$ if $\mu\in F^\times$), is given by
 $$N_D(x) = N_K(x_0) + \mu N_K(x_1) +\mu^2 N_K(x_2) - \mu T_{K}(x_0 \sigma(x_1) \sigma^2(x_2)).$$
 If $N_D$ is anisotropic then $D$ is a division algebra over $F$.
If $\mu\in K\setminus F$ we  get $N_D(lx) = N_K(l)N_D(x)$  for all $x \in D$, $l \in K$ \cite[Proposition 4.2.2, 4.2.3]{St2}.

  On the other hand, $J(K, \mu)$  is a nine-dimensional nonassociative unital algebra over $F$ with multiplication
$$x  y = (x_0 \cdot y_0 + \overline{x_1y_2} + \overline{x_2y_1}, \overline{x_0}y_1 + \overline{y_0}x_1 +  \mu^{-1}(x_2 \times y_2),
 \overline{x_0}y_2 + \overline{y_0}x_2 +  \mu (x_1 \times y_1))$$
 for $x = (x_0, x_1, x_2), y = (y_0, y_1, y_2) \in J(K, \mu )$,
 cubic norm map
\[N((x_0, x_1, x_2)) = N_K(x_0) +  \mu  N_K(x_1) +  \mu^{-1} N_K(x_2) - T_K(x_0x_1x_2),\]
 and  trace $T(x) = T_K(x_{0})$. Moreover, we have
 \[x^\sharp = (\sigma(x_0) \sigma^2(x_0) - x_1x_2,  \mu^{-1}\sigma(x_2) \sigma^2(x_2) - x_0x_1,  \mu  \sigma(x_1) \sigma^2(x_1) - x_2x_0).\]
If $ \mu \in F^{\times} $ then  $ D = (K/F,\sigma,\mu) $ is an associative cyclic algebra over $F$ of degree three and
 $J(K, \mu)\cong D^+$ is a special cubic Jordan algebra.
 It is well known that the isomorphism $G: D^+=(K/F,\sigma,\mu)^+ \rightarrow J(K,\mu)$ is given by
 $$   x_{0}+x_{1}z+x_{2}z^{2} \mapsto (x_{0},\sigma(x_{1}),\mu\sigma^{2}(x_{2})).$$
 However, if $ \mu \in K \setminus F$, $G$ is not an algebra isomorphism between $ (K/F,\sigma,\mu)^{+} $ and $ J(K,\mu) $, where now $ (K/F,\sigma,\mu) $ is a nonassociative cyclic algebra, since $ \sigma(\mu) \neq \mu $.
However, for $ \mu \in K \setminus F$, the map $G: D^+ \rightarrow J(K,\mu)$ still yields an isometry of norms,
 since
 $$N((x_0, \sigma(x_{1}),\mu\sigma^{2}(x_{2})) = N_K(x_0) +  \mu  N_K(x_1) +  \mu^{-1} N_K(x_2) - \mu T_K(x_0\sigma(x_1)\sigma(x_2))=
N_D((x_0, x_1, x_2) ),$$
 and hence the norms of the two nonisomorphic algebras $ D^+=(K/F,\sigma,\mu)^+$ and $ J(K,\mu)$ are isometric.
\\\\
 \emph{Statements and Declarations}
  This paper was written while the second author was a visitor at the University of Ottawa. She acknowledges support from  the Centre de Recherches Math\'ematiques for giving a colloquium talk,
  and from Monica Nevins' NSERC Discovery Grant RGPIN-2020-05020 (the author herself does not hold any grant). She would like to thank the Department of Mathematics and Statistics for its hospitality.


\begin{thebibliography}{[B-C-R]}


\bibitem{BP}  C. Brown,  S. Pumpl\"{u}n, \emph{Nonassociative cyclic extensions of fields and central simple algebras}.
 Journal of Pure and Applied Algebra 223 (6) (2019), 2401-2412.


\bibitem{Jacobson}
N.~Jacobson, ``Structure and representations of Jordan algebras.'' AMS. Providence, AMS. Colloquium Publications, Vol. XXXIX (1968).


\bibitem{KMRT}%{Involutions}
 M. A. Knus, A. Merkurjev, M. Rost, J.-P. Tignol,
``The Book of Involutions.'' AMS Colloquium Publications 44 (1998).


\bibitem{McC}%{McCrimmon}
 K. McCrimmon,  ``A Taste of Jordan Algebras.'' Universitext, Springer Verlag, New York (2004).




%\bibitem{PST} R. Parimala,  V. Suresh,   M. L. Thakur, {\it Jordan algebras and $F_4$ bundles over the affine
%plane}. J. Algebra 198 (1997), 582-607.

%\bibitem{PR1} H.~P.~Petersson, M.~L.~Racine, {\it Jordan algebras of degree three  and the Tits process}.
%J. Algebra 98 (1986) (1), 211-243.


\bibitem{PR2} %{Serre-rost}
H.~P.~Petersson, M.~L.~Racine,
{\it An elementary approach to the Serre-Rost invariant of Albert algebras\/}, Indag. Math. (N.S.) 7, 343-365 (1996).


\bibitem{PR2} H.~P.~Petersson, M.~L.~Racine, {\it Classification of algebras arising from the Tits process}.
J. Algebra 98 (1986) (1), 244-279.


\bibitem{PR} Petersson, H. P., Racine, M. L., {\it Jordan algebras of degree 3 and the Tits process}.
J. Algebra 98 (1986) (1), 211-243.


\bibitem{Pu0} S. Pumpl\"{u}n,  \emph{The automorphisms of generalized cyclic Azumaya algebras.}
J. Pure Applied Algebra 225 (4) (2021).
\verb#https://doi.org/10.1016/j.jpaa.2020.106540#

%\bibitem{Pu14} S. Pumpl\"{u}n,   {\it How to obtain division algebras from a generalized Cayley-Dickson doubling process,}
 %J. Algebra 404 (2014), 406-434.


\bibitem{Pu-U} S. Pumpl\"{u}n, T. Unger,  {\it Space-time block codes from nonassociative division algebras.}
 Advances in Mathematics of Communications 5 (3), 609-629.

\bibitem{S} R. Sandler,  {\it Autotopism groups of some finite nonassociative algebras}. AMS Journal of Mathematics 84 (2)
(1962), 239-264.

\bibitem{Sch}% {Richard Schafer}
 R. D. Schafer,  {\it Forms permitting composition}. Advances in Mathematics 4 (1970), 127-148.

\bibitem{St1} %{Steele2}
 A. Steele,   {\it Nonassociative cyclic algebras}.  Israel Journal of Mathematics 200 (2014), 361–387.

\bibitem{St2}%{Steele}
 A. Steele,  {\it Some new classes of algebras}. PhD Thesis, University of Nottingham 2013.
 \verb#http://eprints.nottingham.ac.uk/13934/1/PhdthesisFinal.pdf#

\bibitem{S-Pu-O} A. Steele, S. Pumpl\"{u}n, F. Oggier,  {\it MIDO space-time codes from associative and non-associative cyclic algebras.}
 Information Theory Workshop (ITW) 2012 IEEE (2012), 192-196.

%\bibitem{T}  M. L. Thakur,  {\it Cayley algebra bundles on $\mathbb{A}_K^2 $ revisited.} Communications in Algebra
%23 (13) (1995), 5119-5130.

\bibitem{Thakur}
M.~Thakur,
{\it Automorphisms of Albert algebras and a conjecture of Tits and Weiss,\/}, Transactions of the American Mathematical Society \emph{365}, 3041–3068 (2013).

\end{thebibliography}
\end{document}